\documentclass[a4paper,12pt]{article}
\usepackage{mathrsfs,amsthm,graphicx,color,verbatim,bm,bbm,amsmath,amsfonts,amssymb,newclude,nicefrac,amsfonts,graphicx,enumerate,hyperref}
\usepackage[latin1]{inputenc}

\newtheorem{lemma}{Lemma}[section]
\newtheorem{theorem}[lemma]{Theorem}

\newtheorem{corollary}[lemma]{Corollary}

\providecommand{\N}{{\ensuremath{\mathbbm{N}}}}
\providecommand{\R}{{\ensuremath{\mathbbm{R}}}}
\providecommand{\E}{{\ensuremath{\mathbb{E}}}}
\renewcommand{\P}{{\ensuremath{\mathbb{P}}}}
\providecommand{\1}{{\ensuremath{\mathbbm{1}}}}

\newcommand{\diffns}[1]{d#1}

\newcommand{\lpnb}[3]{L^{#1}(#2;#3)}
\newcommand{\smallsum}{\textstyle\sum}
\newcommand{\smallprod}{\textstyle\prod}

\newcommand{\stochval}[1]{|\![#1]\!|}
\newcommand{\nzspace}[1]{{#1}\setminus\{0\}}

\title{Counterexamples to regularities for the derivative processes associated to stochastic evolution equations}

\author{Mario Hefter, Arnulf Jentzen, and Ryan Kurniawan}

\begin{document}

\maketitle

\begin{abstract}

In the recent years there has been an increased interest in studying regularity properties of the derivatives of stochastic evolution equations (SEEs) with respect to their initial values.
In particular, in the scientific literature it has been shown for every natural number $n\in\N$ that if the nonlinear drift coefficient and the nonlinear diffusion coefficient of the considered SEE are $n$-times continuously Fr\'{e}chet differentiable, 
then the solution of the considered SEE is also $n$-times continuously Fr\'{e}chet differentiable with respect to its initial value and the corresponding derivative processes satisfy a suitable regularity property in the sense that the $n$-th derivative process can be extended continuously to
$n$-linear operators on negative Sobolev-type spaces with regularity parameters $\delta_1,\delta_2,\ldots,\delta_n\in[0,\nicefrac{1}{2})$ provided that the condition
$
  \sum^n_{i=1}
  \delta_i
  < \nicefrac{1}{2}
$ 
is satisfied.
The main contribution of this paper is to reveal that this condition 
can essentially not be relaxed.
\end{abstract}

\tableofcontents

\section{Introduction}
In the recent years there has been an increased interest in studying regularity properties of the derivatives of stochastic evolution equations (SEEs) with respect to their initial values (cf., e.g., Cerrai~\cite[Chapters~6--7]{c01},  Debussche~\cite[Lemmas~4.4--4.6]{Debussche2011}, Wang \& Gan~\cite[Lemma~3.3]{WangGan2013_Weak_convergence}, Andersson et al.~\cite{AnderssonJentzenKurniawan2016arXiv,AnderssonJentzenKurniawan2016a}).
One important reason for this increased interest is that appropriate estimates on the first, second, and higher order derivatives of SEEs with respect to their initial values have been used as key tools for establishing essentially sharp weak convergence rates (see, e.g., Debussche~\cite[Theorem~2.2]{Debussche2011},
Wang \& Gan~\cite[Theorem~2.1]{WangGan2013_Weak_convergence}, 
Andersson \& Larsson~\cite[Theorem~1.1]{AnderssonLarsson2015},
Br\'{e}hier~\cite[Theorem~1.1]{Brehier2014},
Br\'{e}hier \& Kopec~\cite[Theorem~5.1]{BrehierKopec2016},
Wang~\cite[Corollary~1]{Wang2016481},
Conus et al.~\cite[Corollary~5.2]{ConusJentzenKurniawan2014arXiv}, 
\cite[Corollary~8.2]{JentzenKurniawan2015arXiv}, 
and Hefter et al.~\cite[Theorem~1.1]{HefterJentzenKurniawan2016}).
In particular, in the recent article Andersson et al.~\cite{AnderssonJentzenKurniawan2016a} it has been shown that if the nonlinear drift coefficient and the nonlinear diffusion coefficient of an SEE are $n$-times continuously Fr\'{e}chet differentiable, 
then the solution of the considered SEE is also $n$-times continuously Fr\'{e}chet differentiable with respect to its initial value and the corresponding derivative processes satisfy a suitable regularity property (see item~(iv) of Theorem~1.1 in Andersson et al.~\cite{AnderssonJentzenKurniawan2016a} and item~\eqref{item:derivative.strong.bound} of Corollary~\ref{cor:positive.result} below, respectively).
In this work we reveal that this regularity property can essentially not be improved.
To illustrate our result in more detail we consider the following notation throughout the rest of this introductory section.
For every measure space $ ( \Omega , \mathcal{F}, \mu ) $,
every measurable space $ ( S , \mathcal{S} ) $,
and every $\mathcal{F}$/$\mathcal{S}$-measurable function
$ X \colon \Omega\to S $
we denote by
$
\left[ X \right]_{
	\mu, \mathcal{S}
}
$
the set given by
\begin{multline}
\left[ X \right]_{
	\mu, \mathcal{S}
}
=
\big\{
Y\colon\Omega\to S
\colon
(Y \text{ is }\mathcal{F}\text{/}\mathcal{S}\text{-measurable})
\\
\wedge
(
\exists \, A \in \mathcal{F} \colon
\mu(A) = 0 
\text{ and }
\{ \omega \in \Omega \colon X(\omega) \neq Y(\omega) \}
\subseteq A
)
\big\}
.
\end{multline}
We first briefly review the above mentioned regularity result on derivative processes of SEEs from Andersson et al.~\cite{AnderssonJentzenKurniawan2016a}.
More formally, Theorem~1.1 in Andersson et al.~\cite{AnderssonJentzenKurniawan2016a} includes the following result, Corollary~\ref{cor:positive.result} below, as a special case.
\begin{corollary}
	\label{cor:positive.result}
For every real number $T\in(0,\infty)$,  
all nontrivial separable $\R$-Hilbert spaces 
$ ( H, \left\| \cdot \right\|_H, \langle \cdot , \cdot \rangle_H ) $
and 
$ ( U, \left\| \cdot \right\|_U, \langle \cdot , \cdot \rangle_U ) $,
every probability space $(\Omega,\mathcal{F},\P)$, 
every normal filtration
$(\mathcal{F}_t)_{t\in[0,T]}$ on $(\Omega,\mathcal{F},\P)$, 
every $\operatorname{Id}_U$-cylindrical 
$ ( \Omega , \mathcal{F}, \P, ( \mathcal{F}_t )_{ t \in [0,T] } ) $-Wiener process $(W_t)_{t\in[0,T]}$, 
every generator 
$A\colon D(A)\subseteq H \to H$ 
of a strongly continuous analytic semigroup with 
$
\operatorname{spectrum}(A)
\subseteq
\{z\in\mathbb{C}\colon\operatorname{Re}(z)<0\}
$, 
and all infinitely often Fr\'{e}chet differentiable functions 
$ F \colon H \to H  $ 
and 
$ B \colon H \to HS(U,H) $
with globally bounded derivatives it holds
\begin{enumerate}[(i)]
	\item
	that there exist up-to-modifications unique
	$ ( \mathcal{F}_t )_{ t \in [0,T] } $/$ \mathcal{B}(H) $-predictable stochastic processes
	$ X^x \colon 
	[0,T] \times \Omega \to H $, 
	$ x \in H $, 
	which fulfill for all
	$ x \in H $, 
	$ p \in [2,\infty) $, 
	$ t \in [0,T] $
	that
$
\int^t_0
\|e^{(t-s)A} F(X^x_s)\|_H
+
\|e^{(t-s)A} B(X^x_s)\|^2_{HS(U,H)}
\, ds
< \infty
$, 
	$
	\sup_{ s \in [0,T] }
	\E\big[\|
	X_s^x
	\|^p_H
	\big]
	< \infty
	$, 
	and
	\begin{equation}
	\label{eq:SEE.positive}
	\begin{split}
	&
	[
	X_t^x
	-
	e^{tA} x
	]_{ \P, \mathcal{B}(H) }
	=
	\left[
	\int_0^t
	e^{ ( t - s ) A }
	F(X_s^x)
	\, \diffns s
	\right]_{ \P, \mathcal{B}(H) }
	+
	\int_0^t
	e^{ ( t - s ) A }
	B(X_s^x)
	\, \diffns W_s
	,
	\end{split}
	\end{equation}
	\item
	that it holds for all 
	$ p \in [2,\infty) $, 
	$ t \in [0,T] $ 
	that 
	$
	H \ni x \mapsto [X^x_t]_{\P,\mathcal{B}(H)} \in \lpnb{p}{\P}{H}
	$ 
	is infinitely often Fr\'{e}chet differentiable, and
\item
\label{item:derivative.strong.bound}
that it holds for all 
$ p \in [2,\infty) $, 
$ n \in \N = \{1,2,\ldots\} $, 
$q\in[0,\infty)$, 
$ \delta_1, \delta_2, \ldots, \delta_n \in [0,\nicefrac{1}{2}) $, 
$ t \in (0,T] $
with 
$
\sum^n_{i=1}
\delta_i
< \nicefrac{1}{2}
$
that 
\begin{equation}
\label{eq:derivative.strong.bound}
\sup_{ x \in H }
\sup_{ u_1,u_2,\ldots,u_n \in \nzspace{H} }
\left[
\frac{
  \big(
  \E\big[\|
    (-A)^{-q} 
    (\frac{d^n}{dx^n}[X^x_t]_{\P,\mathcal{B}(H)})(u_1,u_2,\ldots,u_n)
  \|^p_H\big]
  \big)^{\nicefrac{1}{p}}
}{  
\prod_{ i = 1 }^n
\left\| (-A)^{-\delta_i} u_i \right\|_H
}
\right]
<
\infty
.
\end{equation}
\end{enumerate}
\end{corollary}

Item~(iv) of Theorem~1.1 in Andersson et al.~\cite{AnderssonJentzenKurniawan2016a}
and item~\eqref{item:derivative.strong.bound} of Corollary~\ref{cor:positive.result} in this paper, respectively, prove that the condition 
\begin{equation}
\label{eq:delta.condition}
  \sum^n_{i=1}
  \delta_i
  < \nicefrac{1}{2}
\end{equation}
for the regularity parameters 
$
  \delta_1,\delta_2,\ldots,\delta_n
  \in [0,\nicefrac{1}{2})
$
of the considered negative Sobolev-type spaces is sufficient to ensure that the left-hand side of~\eqref{eq:derivative.strong.bound} is finite.
The main result of this work (see Corollary~\ref{cor:negative.result} below and Theorem~\ref{lem:blowup} in Subsection~\ref{sec:disprove} below, respectively) reveals that this condition can essentially not be relaxed.
More specifically, Theorem~\ref{lem:blowup}  in Subsection~\ref{sec:disprove} below directly implies the following result.

\begin{corollary}
	\label{cor:negative.result}
	For every real number $T\in(0,\infty)$,
	every infinite dimensional separable $\R$-Hilbert space
	$ ( H, \left\| \cdot \right\|_H, \langle \cdot , \cdot \rangle_H ) $,
	every nontrivial separable $\R$-Hilbert space 
	$ ( U, \left\| \cdot \right\|_U, \langle \cdot , \cdot \rangle_U ) $, 
	every probability space $(\Omega,\mathcal{F},\P)$, 
	every normal filtration
	$(\mathcal{F}_t)_{t\in[0,T]}$ on $(\Omega,\mathcal{F},\P)$, 
		and every $\operatorname{Id}_U$-cylindrical 
		$ ( \Omega , \mathcal{F}, \P, ( \mathcal{F}_t )_{ t \in [0,T] } ) $-Wiener process $(W_t)_{t\in[0,T]}$ 
	there exist a generator $A\colon D(A)\subseteq H \to H$ of a strongly continuous analytic semigroup with 
	$
	\operatorname{spectrum}(A)
	\subseteq
	\{z\in\mathbb{C}\colon\operatorname{Re}(z)<0\}
	$
	and infinitely often Fr\'{e}chet differentiable functions 
	$ F \colon H \to H  $ 
	and 
	$ B \colon H \to HS(U,H) $
	with globally bounded derivatives 
	such 
	\begin{enumerate}[(i)]
		\item
		that there exist up-to-modifications unique
		$ ( \mathcal{F}_t )_{ t \in [0,T] } $/$ \mathcal{B}(H) $-predictable stochastic processes
		$ X^x \colon 
		$
		$
		[0,T] \times \Omega \to H $, 
		$ x \in H $, 
		which fulfill for all
		$ x \in H $, 
		$ p \in [2,\infty) $, 
		$ t \in [0,T] $
		that
		$
		\int^t_0
		\|e^{(t-s)A} F(X^x_s)\|_H
		+
		\|e^{(t-s)A} B(X^x_s)\|^2_{HS(U,H)}
		\, ds
		< \infty
		$, 
		$
		\sup_{ s \in [0,T] }
		\E\big[\|
		X_s^x
		\|^p_H
		\big]
		< \infty
		$, 
		and
		\begin{equation}
		\label{eq:SEE.negative}
		\begin{split}
		&
		[
		X_t^x
		-
		e^{tA} x
		]_{ \P, \mathcal{B}(H) }
		=
		\left[
		\int_0^t
		e^{ ( t - s ) A }
		F(X_s^x)
		\, \diffns s
		\right]_{ \P, \mathcal{B}(H) }
		+
		\int_0^t
		e^{ ( t - s ) A }
		B(X_s^x)
		\, \diffns W_s
		,
		\end{split}
		\end{equation}
		\item
		that it holds for all 
		$ p \in [2,\infty) $, 
		$ t \in [0,T] $ 
		that 
		$
		H \ni x \mapsto [X^x_t]_{\P,\mathcal{B}(H)} \in \lpnb{p}{\P}{H}
		$ 
		is infinitely often Fr\'{e}chet differentiable,
		\item
		that it holds for all 
		$ p \in [2,\infty) $, 
		$ n \in \N $, 
		$q\in[0,\infty)$, 
		$ \delta_1, \delta_2, \ldots, \delta_n \in [0,\nicefrac{1}{2}) $, 
		$ t \in (0,T] $
		with 
		$
		\sum^n_{i=1}
		\delta_i
		< \nicefrac{1}{2}
		$
		that 
		\begin{equation}
		\sup_{ x \in H }
		\sup_{ u_1,u_2,\ldots,u_n \in \nzspace{H} }
		\left[
		\frac{
			\big(
			\E\big[\|
			(-A)^{-q} 
			(\frac{d^n}{dx^n}[X^x_t]_{\P,\mathcal{B}(H)})(u_1,u_2,\ldots,u_n)
			\|^p_H\big]
			\big)^{\nicefrac{1}{p}}
		}{  
			\prod_{ i = 1 }^n
			\left\| (-A)^{-\delta_i} u_i \right\|_H
		}
		\right]
		<
		\infty
		,
		\end{equation}
		and
		\item
		\label{item:intro.regularity}
that it holds for all 
$ p \in [2,\infty) $, 
$ n \in \N $, 
$q\in[0,\infty)$, 
$ \delta_1, \delta_2, \ldots, \delta_n \in \R $, 
$ t \in (0,T] $ 
with 
$
\sum^n_{i=1}
\delta_i
> \nicefrac{1}{2}
$
that 
\begin{equation}
\sup_{ x,u_1,u_2,\ldots,u_n \in \nzspace{(\cap_{r\in\R}H_r)} }
\left[
\frac{
	\big(
	\E\big[\|
	(-A)^{-q}
	(\frac{d^n}{dx^n}[X^x_t]_{\P,\mathcal{B}(H)})(u_1,u_2,\ldots,u_n)
	\|^p_H\big]
	\big)^{\nicefrac{1}{p}}
}{  
\prod_{ i = 1 }^n
\left\| (-A)^{-\delta_i} u_i \right\|_H
}
\right]
=
\infty
.
\end{equation}
\end{enumerate}
\end{corollary}
\noindent Regularity results for Kolmogorov equations associated to SEEs of the form~\eqref{eq:SEE.positive} and~\eqref{eq:SEE.negative}, which are in some sense related to Corollaries~\ref{cor:positive.result} and~\ref{cor:negative.result}, can, e.g., be found in Debussche~\cite[Lemmas~4.4--4.6]{Debussche2011},
Wang \& Gan~\cite[Lemma~3.3]{WangGan2013_Weak_convergence}, 
Andersson \& Larsson~\cite[(4.2)--(4.3)]{AnderssonLarsson2015},
Br\'{e}hier~\cite[Propositions~5.1--5.2 and Lemma~5.4]{Brehier2014},
Wang~\cite[Lemma~3.3]{Wang2016481},
Andersson et al.~\cite[Theorem~3.3]{AnderssonHefterJentzenKurniawan2016}, 
and Brehier \& Debussche~\cite[Theorems~3.2--3.3 and Proposition~3.5]{BrehierDebussche2017}.

The remainder of this article is organized as follows.
In Section~\ref{sec:counterexample} we state and prove the main result of this paper; see Theorem~\ref{lem:blowup} in Subsection~\ref{sec:disprove} below.
In Subsection~\ref{sec:counter.setting} we present the drift and the diffusion coefficient functions that we use throughout Section~\ref{sec:counterexample}.
In Subsection~\ref{sec:counter} we derive an explicit representation of the considered diffusion coefficient function (see Lemma~\ref{lem:derivative.formula} in Subsection~\ref{sec:counter}). 
In Subsection~\ref{sec:derivative.process} we present explicit formulas for the solution and its derivatives of the SEE associated with the drift and diffusion coefficient functions considered in Subsection~\ref{sec:counter.setting} (see Lemma~\ref{lem:derivative.process} in Subsection~\ref{sec:derivative.process}).
In Subsection~\ref{sec:disprove} we employ Lemma~\ref{lem:derivative.formula} in Subsection~\ref{sec:counter} and Lemma~\ref{lem:derivative.process} in Subsection~\ref{sec:derivative.process} to prove the main result of this paper, Theorem~\ref{lem:blowup} in Subsection~\ref{sec:disprove}.
Corollary~\ref{cor:negative.result} above is an immediate consequence of Theorem~\ref{lem:blowup} in Subsection~\ref{sec:disprove}.

\section{Counterexamples to regularities for the derivative processes associated to stochastic evolution equations}
\label{sec:counterexample}

\subsection{Setting}
\label{sec:counter.setting}

Throughout this section we consider the following setting. 
For every set $A$ let $ \mathcal{P}(A) $ be the power set of $ A $
and let $ \#_A \in \N_0 \cup \{\infty\} $ 
be the number of elements of $ A $,
let 
$ \Pi_k \in
\mathcal{P}(\mathcal{P}(
\mathcal{P}( \N )
)) 
$, 
$ k \in \N_0 $,
be the sets which satisfy for all $ k \in \N $
that
$\Pi_0=\emptyset$
and
\begin{equation}
\Pi_k =
\left\{
A \subseteq \mathcal{P}( \N )
\colon
\left[
\emptyset \notin A
\right]
\wedge
\left[
\cup_{ a \in A }
a
=
\left\{ 1, 2, \dots, k \right\}
\right]
\wedge
\left[
\forall \, a, b \in A \colon
\left(
a \neq b
\Rightarrow
a \cap b = \emptyset
\right)
\right]
\right\}
\end{equation}
(see, e.g., (10) in Andersson et al.~\cite{AnderssonJentzenKurniawan2016a}),
let 
$ ( H, \left\| \cdot \right\|_H, \langle \cdot , \cdot \rangle_H ) $
be an $ \R $-Hilbert space,
let $ e = ( e_n )_{ n \in \N } \colon \N \to H $ be an orthonormal basis of $ H $, 
let $ \lambda = ( \lambda_n )_{ n \in \N } \colon \N \to \R $, 
$ P \colon H \to H $, 
and $ B \colon H \to H $ 
be functions which satisfy for all 
$ v \in H $
that
$ \sup_{ n \in \N } \lambda_n < 0 $, 
$
  P v = \sum_{ n = 2 }^{ \infty } \langle e_n, v \rangle_H e_n
$, 
and 
$
  B( v ) =
  \sqrt{
    1 + \| P v \|^2_H
  } \,
  e_1
$, 
let $ T \in (0,\infty) $, 
let 
$ ( \Omega, \mathcal{F}, \P ) $
be a probability space with a normal filtration 
$(\mathcal{F}_t)_{t\in[0,T]}$,
let $ W \colon [0,T] \times \Omega \to \R $
be a standard 
$ ( \Omega , \mathcal{F}, \P, ( \mathcal{F}_t )_{ t \in [0,T] } ) $-Brownian motion,
let $ A \colon D(A) \subseteq H \rightarrow H $ 
be the linear operator which satisfies
$
  D(A)
  =
  \{ 
    v \in H \colon 
    \sum^\infty_{n=1}
    \left|
    \lambda_n 
    \langle e_n,v \rangle_H
    \right|^2
    < \infty
  \}
$ 
and 
$
  \forall \, v \in D(A) \colon
  Av
  =
  \sum^\infty_{n=1}
  \lambda_n 
  \langle e_n,v\rangle_H
  e_n
$, 
let
$
  (
    H_r
    ,
    \left\| \cdot \right\|_{ H_r }
    ,
    \langle \cdot , \cdot \rangle_{ H_r }
  )
$,
$ r \in \R $,
be a family of interpolation spaces associated to
$
  - A
$
(cf., e.g., \cite[Section~3.7]{sy02}),
and 
for every $\mathcal{F}$/$\mathcal{B}(H)$-measurable function
$ X \colon \Omega\to H $
let 
$
  \stochval{X}
$
be the set given by
$
  \stochval{X}
  =
  \big\{ Y \colon \Omega \to H 
  \colon 
  ( Y \text{ is }\mathcal{F}\text{/}\mathcal{B}(H)\text{-}
$
$  
  \text{-measurable and } \P(X=Y) = 1  )
  \big\}
$.

\subsection{An explicit representation for the diffusion coefficient}
\label{sec:counter}

\begin{lemma} [Derivatives of the diffusion $B$]
\label{lem:derivative.formula}
Assume the setting in Section~\ref{sec:counter.setting}.
Then 
\begin{enumerate}[(i)]
\item 
\label{item:B.smoothness}
it holds that 
$
  B \colon H \to H
$ 
is infinitely often differentiable,
\item
\label{item:B.formula}
it holds for all 
$ n \in \N $, 
$
  v_0, v_1, \ldots, v_n \in H
$ 
that 
\begin{equation}
\label{eq:derivative.formula}
\begin{split}
&
  B^{(n)}( v_0 )( v_1,v_2,\ldots,v_n )
\\&=
  \Bigg(
  {\sum\limits_{
    \varpi \in \Pi_n
  }}
  \frac{
    \big[
      {\prod^{ \#_\varpi - 1 }_{ i=0 }}
      (1-2i)
    \big]
  \big[
  \prod\nolimits_{ I \in \varpi }
  \langle
    \mathbbm{1}_{\{1,2\}}( \#_I ) \, Pv_{\max(I) \1_{\{2\}}(\#_I)} 
    ,
    v_{ \min(I) }
  \rangle_H
  \big]
  }{
  \big[ 1 + \| Pv_0 \|^2_H \big]^{
    ( \#_\varpi - \nicefrac{1}{2} )
  }
  }
  \Bigg) \,
  e_1
  ,
\end{split}
\end{equation}
and 
\item
\label{item:B.derivative.bounded}
it holds for all $ n \in \N $ that 
$
  \sup_{v\in H}
  \|B^{(n)}(v)\|_{L^{(n)}(H,H)}
  < \infty
$.
\end{enumerate}
\end{lemma}
\begin{proof}
Throughout this proof let 
$ f \in C^\infty( (0,\infty), \R ) $ 
and 
$ g \in C^\infty( H, (0,\infty) ) $ 
be the functions which satisfy for all 
$ x \in (0,\infty) $, $ v \in H $ 
that 
\begin{equation}
  f(x) = \sqrt{x}
  \qquad\text{and}\qquad
  g(v) = 1+\| Pv \|^2_H
\end{equation}
and let
$
  I^\varpi_i
  \in
  \varpi
$, 
$i \in \{1,2,\ldots,\#_\varpi\} $, 
$ \varpi \in \Pi_n $, 
$ n \in \N $, 
be
the sets which satisfy for all 
$ n \in \N $, 
$ \varpi \in \Pi_n $
that
\begin{equation}
  \min( I^\varpi_1 ) < 
  \min( I_2^\varpi ) < \dots < 
  \min( 
    I_{ \#_\varpi }^{ \varpi } 
  )
  .
\end{equation}
Note that the fact that 
$
  \forall\, v \in H
  \colon
  B(v)
  =
  (f\circ g)(v) \, e_1
  =
  f(g(v))\,
  e_1
$ 
proves item~\eqref{item:B.smoothness}.

In the next step we prove~\eqref{eq:derivative.formula} by induction on $n\in\N$.
For the base case $ n=1 $ we note that for all 
$ v_0, v_1 \in H $ 
it holds that
\begin{equation}
\begin{split}
  B'(v_0) v_1
&=
  \big[
  (f \circ g)'(v_0) v_1 
  \big] e_1
  =
  \big[
  (f' \circ g)(v_0) g'(v_0) v_1 
  \big] e_1
\\&=
  \frac{1}{
    [1+\|P v_0\|^2_H]^{\nicefrac{1}{2}}
  }
  \langle P v_0, P v_1 \rangle_H \,
  e_1
=
  \frac{
  	\langle P v_0, v_1 \rangle_H
  	}{
  	[1+\|P v_0\|^2_H]^{\nicefrac{1}{2}}
  }\,
  e_1  
  .
\end{split}
\end{equation}
This and the fact that $\Pi_1 = \{\{\{1\}\}\}$ prove~\eqref{eq:derivative.formula} in the base case $n=1$. 
For the induction step $ \N \ni n \to n+1 \in \{2,3,\ldots\} $
assume that~\eqref{eq:derivative.formula} holds for some natural number $ n \in \N $.
Observe that item~\eqref{item:B.smoothness}, the induction hypothesis, and the product rule of differentiation ensure that for all 
$ v_0, v_1, \ldots, v_{n+1} \in H $ 
it holds that 
\begin{equation}
\begin{split}
&
  B^{(n+1)}( v_0 )( v_1,v_2,\ldots,v_{n+1} )
\\&=
  \big(
  \tfrac{d}{d v_0} 
  \big[
  B^{(n)}( v_0 )( v_1,v_2,\ldots,v_n ) 
  \big]\big)
  v_{n+1}
\\&=
  \Bigg(
  \sum_{
    \varpi \in \Pi_n
  }
    \left[
      {\smallprod^{ \#_\varpi - 1 }_{ j=0 }}
      (1-2j)
    \right]
  \Bigg(
  \frac{d}{d v_0}
  \Bigg[
  \frac{
  \prod_{ I \in \varpi }
  \langle
    \mathbbm{1}_{\{1,2\}}( \#_I ) \, Pv_{\max(I) \1_{\{2\}}(\#_I)} 
    ,
    v_{ \min(I) }
  \rangle_H
  }{
    [ 1 + \| Pv_0 \|^2_H ]^{
      ( \#_\varpi - \nicefrac{1}{2} )
    }
  }
  \Bigg]
  \Bigg) v_{n+1}
  \Bigg) \,
  e_1
\\&=
\Bigg(
\sum_{
	\substack{\varpi \in \Pi_n, \\ \forall \, I \in \varpi \colon \#_I\leq 2}
}
\left[
{\smallprod^{ \#_\varpi - 1 }_{ j=0 }}
(1-2j)
\right]
\Bigg(
\frac{d}{d v_0}
\Bigg[
\frac{
	\prod^{\#_\varpi}_{i=1}
	\langle
	Pv_{\max(I^\varpi_i) \1_{\{2\}}(\#_{I^\varpi_i})} 
	,
	v_{ \min(I^\varpi_i) }
	\rangle_H
}{
[ 1 + \| Pv_0 \|^2_H ]^{
	( \#_\varpi - \nicefrac{1}{2} )
}
}
\Bigg]
\Bigg) v_{n+1}
\Bigg) \,
e_1
\\&=
\Bigg(
\sum_{
	\substack{\varpi \in \Pi_n, \\ \forall \, I \in \varpi \colon \#_I\leq 2}
}
\Bigg\{
\frac{
	\big[
	\prod^{ \#_\varpi }_{ j=0 }
	(1-2j)
	\big]
	\langle P v_0, P v_{n+1} \rangle_H
	\big[
	\prod^{ \#_\varpi }_{ i=1 }
	\langle
	Pv_{\max(I^\varpi_i) \1_{\{2\}}(\#_{I^\varpi_i})} 
	,
	v_{ \min(I^\varpi_i) }
	\rangle_H
	\big]
}{
[ 1 + \| Pv_0 \|^2_H ]^{
	( \#_\varpi + \nicefrac{1}{2} )
}
}
\\&\quad+
\frac{
	\big[\prod^{\#_\varpi-1}_{j=0}(1-2j)\big]
}{
[1+\|P v_0\|^2_H]^{(\#_\varpi-\nicefrac{1}{2})}
}
\Bigg(
\frac{d}{dv_0}
\Bigg[
\prod^{\#_\varpi}_{i=1}
\langle
Pv_{\max(I^\varpi_i) \1_{\{2\}}(\#_{I^\varpi_i})} 
,
v_{ \min(I^\varpi_i) }
\rangle_H	
\Bigg]
\Bigg)
v_{n+1}
\Bigg\}
\Bigg)
e_1
  .
\end{split}
\end{equation}
Hence, we obtain that for all 
$ v_0, v_1, \ldots, v_{n+1} \in H $
it holds that 
\begin{equation}
\begin{split}
&
B^{(n+1)}( v_0 )( v_1,v_2,\ldots,v_{n+1} )
\\&=
\Bigg(
\sum_{
	\substack{\varpi \in \Pi_n, \\ \forall \, I \in \varpi \colon \#_I\leq 2}
}
\Bigg\{
\frac{
	\big[
	\prod^{ \#_\varpi }_{ j=0 }
	(1-2j)
	\big]
	\langle P v_0, P v_{n+1} \rangle_H
	\big[
	\prod_{ I \in \varpi }
	\langle
	Pv_{\max(I) \1_{\{2\}}(\#_I)} 
	,
	v_{ \min(I) }
	\rangle_H
	\big]
}{
[ 1 + \| Pv_0 \|^2_H ]^{
	( \#_\varpi + \nicefrac{1}{2} )
}
}
\\&\quad+
\sum_{i\in\{1,2,\ldots,\#_\varpi\}}  
\frac{
	\big[
	\prod^{ \#_\varpi - 1 }_{ j=0 }
	(1-2j)
	\big]
}{
[ 1 + \| Pv_0 \|^2_H ]^{
	( \#_\varpi - \nicefrac{1}{2} )
}
}\,
\langle
\1_{\{2\}}(\#_{I^\varpi_i\cup\{n+1\}}) \, Pv_{n+1} 
,
v_{ \min(I^\varpi_i) }
\rangle_H
\\&\quad\cdot
\prod_{ j \in \{1,2,\ldots,\#_\varpi\}\setminus\{i\} }
\langle
Pv_{\max(I^\varpi_j) \1_{\{2\}}(\#_{I^\varpi_j})} 
,
v_{ \min(I^\varpi_j) }
\rangle_H
\Bigg\}
\Bigg) \, e_1
\\&=
\Bigg(
\sum_{
	\substack{\varpi \in \Pi_n}
}
\Bigg\{
\frac{
	\big[
	\prod^{ \#_{ \varpi \cup \{ \{n+1\} \} }-1 }_{ j=0 }
	(1-2j)
	\big]
	\big[
	\prod_{ I \in \varpi \cup \{ \{n+1\} \} }
	\langle
	\1_{\{1,2\}}(\#_I)
	Pv_{\max(I) \1_{\{2\}}(\#_I)} 
	,
	v_{ \min(I) }
	\rangle_H
	\big]
}{
[ 1 + \| Pv_0 \|^2_H ]^{
	( \#_{ \varpi \cup \{ \{n+1\} \} } - \nicefrac{1}{2} )
}
}
\\&\quad+
\sum_{i\in\{1,2,\ldots,\#_\varpi\}}  
\frac{
	\big[
	\prod^{ \#_\varpi - 1 }_{ j=0 }
	(1-2j)
	\big]
}{
[ 1 + \| Pv_0 \|^2_H ]^{
	( \#_\varpi - \nicefrac{1}{2} )
}
}\,
\langle
\1_{\{1,2\}}(\#_{I^\varpi_i\cup\{n+1\}}) \, Pv_{n+1} 
,
v_{ \min(I^\varpi_i) }
\rangle_H
\\&\quad\cdot
\prod_{ I \in \varpi\setminus\{I^\varpi_i\} }
\langle
\1_{\{1,2\}}(\#_I)
Pv_{\max(I) \1_{\{2\}}(\#_{I})} 
,
v_{ \min(I) }
\rangle_H
\Bigg\}
\Bigg) \, e_1
.
\end{split}
\end{equation}
This implies that for all 
$ v_0, v_1, \ldots, v_{n+1} \in H $ 
it holds that 
\begin{equation}
\begin{split}
&
  B^{(n+1)}( v_0 )( v_1,v_2,\ldots,v_{n+1} )
\\&=
  \Bigg(
  \sum_{
    \varpi \in \Pi_n
  }
  \Bigg\{
    \sum_{
    \substack{\Xi \in \Pi_{n+1}, \\ \Xi = \varpi \cup \{\{n+1\}\}}
    }
    \Bigg[
    \frac{
      \big[
      \prod^{ \#_\Xi - 1 }_{ i=0 }
      (1-2i)
      \big]
  \big[
  \prod_{ I \in \Xi }
  \langle
    \mathbbm{1}_{\{1,2\}}( \#_I ) \, Pv_{\max(I) \1_{\{2\}}(\#_I)} 
    ,
    v_{ \min(I) }
  \rangle_H
  \big]
    }{
  [ 1 + \| Pv_0 \|^2_H ]^{
    ( \#_\Xi - \nicefrac{1}{2} )
  }
    }
  \Bigg]
\\&+
  \sum_{\substack{ \Xi \in \Pi_{n+1}, \, i \in \{1,2,\ldots,\#_\varpi\}, \\ 
  \Xi = (\varpi\setminus\{I^\varpi_i\})\cup\{I^\varpi_i\cup\{n+1\}\}}}
  \Bigg[
    \frac{
      \big[
      \prod^{ \#_\Xi - 1 }_{ i=0 }
      (1-2i)
      \big]
  \big[
  \prod_{ I \in \Xi }
  \langle
    \mathbbm{1}_{\{1,2\}}( \#_I ) \, Pv_{\max(I) \1_{\{2\}}(\#_I)} 
    ,
    v_{ \min(I) }
  \rangle_H
  \big]
    }{
  [ 1 + \| Pv_0 \|^2_H ]^{
    ( \#_\Xi - \nicefrac{1}{2} )
  }
    }
  \Bigg]
  \Bigg\}
  \Bigg) \,
  e_1.
\end{split}
\end{equation}
Combining this with the fact that 
\begin{equation}
\label{eq:recursive.rep}
\begin{split}
&
\Pi_{n+1}
=
\Big\{ 
\varpi \cup \big\{\{n+1\}\big\}
\colon 
\varpi \in \Pi_n 
\Big\}
\\&
\biguplus
\Big\{
\big\{
I^\varpi_1, I^\varpi_2, \ldots, I^\varpi_{i-1}, 
I^\varpi_i \cup \{n+1\}, I^\varpi_{i+1}, I^\varpi_{i+2}, 
\ldots, I^\varpi_{\#_\varpi}
\big\}
\colon
i \in \{1,2,\ldots,\#_\varpi\}, \,
\varpi \in \Pi_n
\Big\}
\end{split}
\end{equation}
proves~\eqref{eq:derivative.formula} in the case $n+1$. 
Induction therefore establishes item~\eqref{item:B.formula}.

It thus remains to prove item~\eqref{item:B.derivative.bounded}.
For this we note that for all 
$ n \in \N $, 
$ \varpi \in \Pi_n $ 
with 
$ \forall \, I \in \varpi \colon \#_I \leq 2 $ 
it holds that 
\begin{equation}
\label{eq:counting.singleton}
  \#_{\{
  	I \in \varpi \colon \#_I = 1
  	\}}
  =
  2\#_\varpi-n
  .
\end{equation}
Next observe that the Cauchy-Schwarz inequality and~\eqref{eq:derivative.formula} ensure that for all 
$ n \in \N $, 
$ v_0, v_1, \ldots, v_n \in H $ 
it holds that 
\begin{equation}
\label{eq:Bnorm.CS}
\begin{split}
&  \| B^{(n)}( v_0 )( v_1,v_2,\ldots,v_n ) \|_{ H }
\\&=
  \Bigg\|
  \Bigg(
  {\sum\limits_{
  		\varpi \in \Pi_n
  	}}
  	\frac{
  		\big[
  		{\prod^{ \#_\varpi - 1 }_{ i=0 }}
  		(1-2i)
  		\big]
  		\big[
  		\prod\nolimits_{ I \in \varpi }
  		\langle
  		\mathbbm{1}_{\{1,2\}}( \#_I ) \, Pv_{\max(I) \1_{\{2\}}(\#_I)} 
  		,
  		v_{ \min(I) }
  		\rangle_H
  		\big]
  	}{
  	\big[ 1 + \| Pv_0 \|^2_H \big]^{
  		( \#_\varpi - \nicefrac{1}{2} )
  	}
  }
  \Bigg) \,
  e_1
  \Bigg\|_H
\\&=
  \Bigg|
  {\sum_{
    \varpi \in \Pi_n, \,
    \forall\, I \in \varpi \colon \#_{I} \leq 2
  }}
  \frac{
    \big[
      {\prod^{ \#_\varpi - 1 }_{ i=0 }}
      (1-2i)
    \big]
  \big[
  \prod\nolimits_{ I \in \varpi }
  \langle
    Pv_{\max(I) \1_{\{2\}}(\#_I)} 
    ,
    Pv_{ \min(I) }
  \rangle_H
  \big]
  }{
  \big[ 1 + \| Pv_0 \|^2_H \big]^{
    ( \#_\varpi - \nicefrac{1}{2} )
  }
  }
  \Bigg|
\\&\leq
  {\sum_{
    \varpi \in \Pi_n, \,
    \forall\, I \in \varpi \colon \#_{I} \leq 2
  }}
  \frac{
    \big|
      {\prod^{ \#_\varpi - 1 }_{ i=0 }}
      (1-2i)
    \big|
  \prod\nolimits_{ I \in \varpi }
  \big[
  \|
    Pv_{\max(I) \1_{\{2\}}(\#_I)} 
  \|_H \,
  \|
    Pv_{ \min(I) }
  \|_H
  \big]
  }{
  \big[ 1 + \| Pv_0 \|^2_H \big]^{
    ( \#_\varpi - \nicefrac{1}{2} )
  }
  }
.
\end{split}
\end{equation}
Moreover, the fact that 
$
  \forall \, n \in \N, \, \varpi \in \Pi_n
  \colon
  \cup_{I\in\varpi} I
  =
  \{1,2,\ldots,n\} 
$
implies that for all 
$ n \in \N $, 
$ \varpi \in \Pi_n $, 
$ v_0, v_1, \ldots, v_n \in H $ 
with 
$
  \forall\, I \in \varpi \colon \#_{I} \leq 2
$
it holds that 
\begin{equation}
\begin{split}
&
\prod\nolimits_{ I \in \varpi }
\big[
\|
Pv_{\max(I) \1_{\{2\}}(\#_I)} 
\|_H \,
\|
Pv_{ \min(I) }
\|_H
\big]
\\&=
\bigg(
\prod\nolimits_{\substack{ I \in \varpi,\\ \#_I = 1 }}
\big[
\|
Pv_0 
\|_H \,
\|
Pv_{ \min(I) }
\|_H
\big]
\bigg) \,
\bigg(
\prod\nolimits_{\substack{ I \in \varpi,\\ \#_I = 2 }}
\big[
\|
Pv_{\max(I)} 
\|_H \,
\|
Pv_{ \min(I) }
\|_H
\big]
\bigg)
\\&=
\bigg(
\prod\nolimits_{\substack{ I \in \varpi,\\ \#_I = 1 }}
\|
Pv_0 
\|_H
\bigg) \,
\bigg\{
\bigg(
\prod\nolimits_{\substack{ I \in \varpi,\\ \#_I = 1 }}
\|
Pv_{\min(I)} 
\|_H
\bigg) \,
\bigg(
\prod\nolimits_{\substack{ I \in \varpi,\\ \#_I = 2 }}
\big[
\|
Pv_{\max(I)} 
\|_H \,
\|
Pv_{ \min(I) }
\|_H
\big]
\bigg)
\bigg\}
\\&=
\|Pv_0\|^{ \#_{\{I\in\varpi\colon\#_I=1\}} }_H \,
\prod^{ n }_{ i=1 }
\| Pv_i \|_H
.
\end{split}
\end{equation}
This, \eqref{eq:counting.singleton}, and~\eqref{eq:Bnorm.CS} show that for all 
$ n \in \N $, 
$ v_0, v_1, \ldots, v_n \in H $ 
it holds that 
\begin{equation}
\begin{split}
&  \| B^{(n)}( v_0 )( v_1,v_2,\ldots,v_n ) \|_{ H }
\\&\leq
\sum_{
	\varpi \in \Pi_n, \,
	\forall\, I \in \varpi \colon \#_{I} \leq 2
}
\left|
\prod^{ \#_\varpi - 1 }_{ i=0 }
(1-2i)
\right|
\frac{
	\|Pv_0\|^{ \#_{\{I\in\varpi\colon\#_I=1\}} }_H
}{
	[ 1 + \| Pv_0 \|^2_H ]^{
		( \#_\varpi - 1/2 )
	}
} 
\left[
\prod^{ n }_{ i=1 }
\| Pv_i \|_H
\right]
\\&=
\sum_{
	\varpi \in \Pi_n, \,
	\forall\, I \in \varpi \colon \#_{I} \leq 2
}
\left|
\prod^{ \#_\varpi - 1 }_{ i=0 }
(1-2i)
\right|
\frac{
	\|Pv_0\|^{ (2\#_\varpi - n) }_H
}{
	[ 1 + \| Pv_0 \|^2_H ]^{
		( \#_\varpi - 1/2 )
	}
} 
\left[
\prod^{ n }_{ i=1 }
\| Pv_i \|_H
\right]
.
\end{split}
\end{equation}
The fact that 
$
  \forall \, v \in H 
  \colon
  \|Pv\|_H \leq \|v\|_H
$
therefore implies that for all 
$ n \in \N $
it holds that 
\begin{equation}
\begin{split}
&
  \sup_{ v \in H }
  \| B^{(n)}( v ) \|_{ L^{(n)}( H, H ) }
\\&\leq
  \sum_{
    \varpi \in \Pi_n, \,
    \forall\, I \in \varpi \colon \#_{I} \leq 2
  }
          \left|
            \prod^{ \#_\varpi - 1 }_{ i=0 }
            (1-2i)
          \right|
  \sup_{ v \in H }
  \bigg[
  \frac{
    \|Pv\|^{ (2\#_\varpi - n) }_H
  }{
  [ 1 + \| Pv \|^2_H ]^{
    ( \#_\varpi - 1/2 )
  }
  }
  \bigg]
\\&\leq
\sum_{
	\varpi \in \Pi_n, \,
	\forall\, I \in \varpi \colon \#_{I} \leq 2
}
\left|
\prod^{ \#_\varpi - 1 }_{ i=0 }
(1-2i)
\right|
\sup_{ v \in H }
\bigg[
\frac{
	[1+\|Pv\|^2_H]^{ (\#_\varpi - \nicefrac{n}{2}) }
}{
[ 1 + \| Pv \|^2_H ]^{
	( \#_\varpi - 1/2 )
}
}
\bigg]
\\&=
  \sum_{
    \varpi \in \Pi_n, \,
    \forall\, I \in \varpi \colon \#_{I} \leq 2
  }
          \left|
            \prod^{ \#_\varpi - 1 }_{ i=0 }
            (1-2i)
          \right|
    \sup_{ v \in H }
    \left[
  \frac{
    1
  }{
  [ 1 + \| Pv \|^2_H ]^{
    ( n-1 )/2
  }
  }
  \right]
\\&\leq
  \sum_{
    \varpi \in \Pi_n, \,
    \forall\, I \in \varpi \colon \#_{I} \leq 2
  }
            \prod^{ \#_\varpi - 1 }_{ i=0 }
            |2i-1|
\leq
  \sum_{\varpi\in\Pi_n}
  [2\#_\varpi]^{\#_\varpi}
    .
\end{split}
\end{equation}
This and the fact that 
$
  \forall \, n \in \N, \, \varpi \in \Pi_n
  \colon
  \#_{\Pi_n}+\#_\varpi < \infty
$
establish item~\eqref{item:B.derivative.bounded}.
The proof of Lemma~\ref{lem:derivative.formula} is thus completed.
\end{proof}

\subsection{Explicit representations for the derivative processes}
\label{sec:derivative.process}

\begin{lemma}[Exact formulas of derivative processes]
\label{lem:derivative.process}
Assume the setting in Section~\ref{sec:counter.setting}. 
Then 
\begin{enumerate}[(i)]
\item
\label{item:base.existence}
there exist up-to-modifications unique
$ ( \mathcal{F}_t )_{ t \in [0,T] } $/$ \mathcal{B}(H) $-predictable stochastic processes
$ X^{0,x} \colon 
$
$
[0,T] \times \Omega \to H $, 
$ x \in H $, 
which fulfill for all
$ p \in [2,\infty) $, 
$ x \in H $, 
$ t \in [0,T] $
that
$
  \sup_{ s \in [0,T] }
  \E\big[
  \|
    X_s^{0,x}
  \|^p_H
  \big]
  < \infty
$
and
\begin{equation}
\label{eq:blow.SEE}
\begin{split}
&
  \stochval{X_t^{0,x}
  	-
  	e^{tA} x}
  =
  \int_0^t
    e^{ ( t - s ) A }
      B(X_s^{0,x})
  \, \diffns W_s
  ,
\end{split}
\end{equation}
\item
\label{item:existence.derivatives}
it holds for all 
$ p \in [2,\infty) $, 
$ t \in [0,T] $ 
that 
$
  H \ni x \mapsto \stochval{X^{0,x}_t} \in \lpnb{p}{\P}{H}
$ 
is infinitely often Fr\'{e}chet differentiable,
\item
\label{item:derivative.predictable}
there exist up-to-modifications unique
$ ( \mathcal{F}_t )_{ t \in [0,T] } $/$ \mathcal{B}(H) $-predictable stochastic processes
$
  X^{n,\mathbf{u}} \colon 
$
$
  [0,T] \times \Omega \to H
$, 
$
  \mathbf{u} \in H^{n+1}
$, 
$
  n \in \N
$,
which fulfill for all 
$ p \in [2,\infty) $, 
$
  n \in \N
$, 
$ \mathbf{u} \in H^n $, 
$
  x \in H
$, 
$ t \in [0,T] $ 
that 
\begin{equation}
\label{eq:SEE.derivatives}
\begin{split}
&
  \big(
  \tfrac{d^n}{dx^n}
  \stochval{X^{0,x}_t}
  \big) \mathbf{u}
  =
  \big(
    H \ni y \mapsto \stochval{X^{0,y}_t} \in \lpnb{p}{\P}{H}
  \big)^{(n)}(x) \, \mathbf{u}
  =
  \stochval{X^{n,(x,\mathbf{u})}_t}
  ,
\end{split}
\end{equation}
\item
\label{item:lem.regularity}
it holds for all 
$ p \in [2,\infty) $, 
$ n \in \N $, 
$ \delta_1,\delta_2,\ldots,\delta_n \in [0,\infty) $, 
$ t \in (0,T] $ 
with 
$ \sum^n_{i=1} \delta_i < \nicefrac{1}{2} $
that
\begin{equation}
\sup_{ \mathbf{u} = (u_0,u_1,\ldots,u_n) \in H \times (\nzspace{H})^n }
\frac{
	\big(\E\big[
	\| X^{ n,\mathbf{u} }_t \|^p_H
	\big]\big)^{\nicefrac{1}{p}}
}{
\prod^{ n }_{ i=1 }
\|u_i\|_{H_{-\delta_i}}
}
< \infty,
\end{equation}
and  
\item
\label{item:derivative.solution}
it holds for all 
$ n \in \N_0 $, 
$ \mathbf{u} = ( u_0, u_1, \ldots, u_n ) \in H^{ n+1 } $, 
$ t \in [0,T] $ 
that 
\begin{equation}
\label{eq:exact.solution}
\begin{split}
&
  \stochval{X_t^{n,\mathbf{u}}
  	-
  	\mathbbm{1}_{ \{ 0, 1 \} }(n) \, e^{tA} u_n  }
  =
  \int_0^t
    e^{ ( t - s ) A }
      B^{ ( n ) }( e^{sA} u_0 )
      \big(
        e^{sA} u_1
        ,
        e^{sA} u_2
        ,
        \dots
        ,
        e^{sA} u_n
      \big)
  \, \diffns W_s
  .
\end{split}
\end{equation}
\end{enumerate}
\end{lemma}
\begin{proof}
Throughout this proof for every 
$ n \in \N $, 
$ \varpi \in \Pi_n $ 
let
$
I^\varpi_i
\in
\varpi
$, 
$i \in \{1,2,\ldots,\#_\varpi\} $, 
be
the sets which satisfy that
\begin{equation}
\min( I^\varpi_1 ) < 
\min( I_2^\varpi ) < \dots < 
\min( 
I_{ \#_\varpi }^{ \varpi } 
),
\end{equation}
let
$
I_{ i, j }^\varpi
\in
I_i^{ \varpi } \subseteq \N
$, 
$ j \in \{ 1,2,\ldots,\#_{I^\varpi_i} \} $, 
$i \in \{1,2,\ldots,\#_\varpi\} $, 
be the natural numbers which satisfy for all 
$i \in \{1,2,\ldots,\#_\varpi\} $
that
\begin{equation}
I_{ i, 1 }^\varpi < I_{ i, 2 }^\varpi < \dots < I_{ i, \#_{ I_i^{ \varpi } } }^\varpi,
\end{equation}
and let 
$
[ \cdot ]_i^\varpi
\colon
H^{ n + 1 }
\to 
H^{ 
	\#_{I_i^\varpi} + 1
}
$, 
$
i \in \{ 1, 2, \dots, \#_\varpi \} 
$, 
be the mappings which satisfy for all
$
i \in \{ 1, 2, \dots, \#_\varpi \} 
$, 
$
\mathbf{u} = (u_0, u_1, \dots, u_n)
\in 
H^{ n + 1 }
$
that
\begin{equation}
[ \mathbf{u} ]_i^\varpi
= ( u_0, u_{ I_{ i, 1 }^\varpi } , u_{ I_{ i, 2 }^\varpi } , \dots , u_{ I_{ i, \#_{I_i^\varpi} }^\varpi } ).
\end{equation} 
We note that items~(i), (ii), (ix), and~(x) of Theorem~2.1 in Andersson et al.~\cite{AnderssonJentzenKurniawan2016a} 
(with 
$T=T$, 
$\eta=0$, 
$H=H$, 
$U=\R$, 
$W=W$, 
$A=A$, 
$F=0$, 
$B=(H\ni v \mapsto (\R\ni u \mapsto B(v)u \in H) \in HS(\R,H))$, 
$\alpha=0$, 
$\beta=0$, 
$ k = n $, 
$ p = p $, 
$ \delta_1=\delta_1 $,
$ \delta_2=\delta_2, \ldots, \delta_n=\delta_n $
for 
$ (\delta_1,\delta_2,\ldots,\delta_n) \in \{(\kappa_1,\kappa_2,\ldots,\kappa_n)\in[0,\nicefrac{1}{2})^n \colon \sum^n_{i=1} \kappa_i < \nicefrac{1}{2}\} $, 
$ p \in [2,\infty) $, 
$ n \in \N $
in the notation of Theorem~2.1 in~\cite{AnderssonJentzenKurniawan2016a})
ensure that
\begin{enumerate}[(a)]
\item
there exist up-to-modifications unique $(\mathcal{F}_t)_{t\in[0,T]}$/$ \mathcal{B}(H) $-predictable stochastic processes
$
X^{ n,\mathbf{u} }
\colon
$
$
[ 0 , T ] \times \Omega
\to H
$, 
$
\mathbf{u} \in H^{n+1}
$, 
$
n \in \N_0
$,
which fulfill
for all
$
n \in \N_0
$,
$ p \in [2,\infty) $,
$
\mathbf{u} = (u_0,u_1,\ldots,u_n) \in H^{n+1}
$, 
$ t \in [0,T] $
that
$
\sup_{s\in[0,T]}
\E\big[\|X^{n,\mathbf{u}}_s\|^p_H\big]
< \infty
$ 
and 
\begin{equation}
\begin{split}
&
\stochval{X_t^{n,\mathbf{u}}
	-
	\mathbbm{1}_{ \{ 0, 1 \} }(n) \, 
	e^{tA} u_n  }
\\&=
\int_0^t
e^{ ( t - s ) A }
\Bigg[
\mathbbm{1}_{ \{ 0 \} }(n)
\,
B(X_s^{0,u_0})
\\&\quad+
\sum_{ \varpi\in \Pi_n }
B^{ ( \#_\varpi ) }( X_s^{ 0,u_0 } )
\big(
X_s^{ \#_{I^\varpi_1}, [ \mathbf{u} ]_1^{ \varpi } }
,
X_s^{ \#_{I^\varpi_2}, [ \mathbf{u} ]_2^{ \varpi } }
,
\dots
,
X_s^{ \#_{I^\varpi_{\#_\varpi}}, [\mathbf{u} ]_{ \#_\varpi }^{ \varpi } }
\big)
\Bigg]
\, \diffns W_s
,
\end{split}
\end{equation}
\item
it holds for all 
$ p \in [2,\infty) $, 
$ t \in [0,T] $ 
that 
$
H \ni x \mapsto \stochval{X^{0,x}_t} \in \lpnb{p}{\P}{H}
$ 
is infinitely often Fr\'{e}chet differentiable,
\item
it holds for all 
$
n \in \N
$, 
$ p \in [2,\infty) $, 
$ \mathbf{u} \in H^n $, 
$
x \in H
$, 
$ t \in [0,T] $ 
that 
\begin{equation}
\begin{split}
&
\big(
\tfrac{d^n}{dx^n}
\stochval{X^{0,x}_t}
\big) \mathbf{u}
=
\big(
H \ni y \mapsto \stochval{X^{0,y}_t} \in \lpnb{p}{\P}{H}
\big)^{(n)}(x) \, \mathbf{u}
=
\stochval{X^{n,(x,\mathbf{u})}_t}
,
\end{split}
\end{equation}
and
\item
it holds for all 
$ n \in \N $, 
$ p \in [2,\infty) $, 
$ \delta_1,\delta_2,\ldots,\delta_n \in [0,\infty) $, 
$ t \in (0,T] $ 
with 
$ \sum^n_{i=1} \delta_i < \nicefrac{1}{2} $
that
\begin{equation}
\sup_{ \mathbf{u} = (u_0,u_1,\ldots,u_n) \in H \times (\nzspace{H})^n }
\frac{
	\big(\E\big[
	\| X^{ n,\mathbf{u} }_t \|^p_H
	\big]\big)^{\nicefrac{1}{p}}
}{
\prod^{ n }_{ i=1 }
\|u_i\|_{H_{-\delta_i}}
}
< \infty.
\end{equation}
\end{enumerate}
This and item~(i) of Corollary~2.10 in Andersson et al.~\cite{AnderssonJentzenKurniawan2016arXiv} establish items~\eqref{item:base.existence}--\eqref{item:lem.regularity}.
It thus remains to prove~\eqref{item:derivative.solution}.
For this let 
$
X^{ n,\mathbf{u} }
\colon
[ 0 , T ] \times \Omega
\to H
$, 
$
\mathbf{u} \in H^{n+1}
$, 
$
n \in \N_0
$,
be $(\mathcal{F}_t)_{t\in[0,T]}$/$ \mathcal{B}(H) $-predictable stochastic processes which fulfill
for all
$
n \in \N_0
$,
$ p \in [2,\infty) $,
$
\mathbf{u} = (u_0,u_1,\ldots,u_n) \in H^{n+1}
$, 
$ t \in [0,T] $
that
$
\sup_{s\in[0,T]}
\E\big[\|X^{n,\mathbf{u}}_s\|^p_H\big]
< \infty
$ 
and 
\begin{multline}
\label{eq:proof.SEE}
\stochval{X_t^{n,\mathbf{u}}
	-
	\mathbbm{1}_{ \{ 0, 1 \} }(n) \, 
	e^{tA} u_n}
=
\int_0^t
e^{ ( t - s ) A }
\Bigg[
\mathbbm{1}_{ \{ 0 \} }(n)
\,
B(X_s^{0,u_0})
\\+
\sum_{ \varpi\in \Pi_n }
B^{ ( \#_\varpi ) }( X_s^{ 0,u_0 } )
\big(
X_s^{ \#_{I^\varpi_1}, [ \mathbf{u} ]_1^{ \varpi } }
,
X_s^{ \#_{I^\varpi_2}, [ \mathbf{u} ]_2^{ \varpi } }
,
\dots
,
X_s^{ \#_{I^\varpi_{\#_\varpi}}, [\mathbf{u} ]_{ \#_\varpi }^{ \varpi } }
\big)
\Bigg]
\, \diffns W_s
.
\end{multline}
Note that~\eqref{eq:proof.SEE} and the fact that 
$
  \forall \, v \in H
  \colon
  P(B(v))
  =
  0
$ 
imply that 
for all 
$ x \in H $, 
$ t \in [0,T] $
it holds that 
\begin{equation}
\label{eq:0.project}
\begin{split}
  \stochval{  P X^{0,x}_t
  	-
  	P e^{tA} x}
  =
  P \int^t_0
  e^{(t-s)A}
  B( X^{0,x}_s )
  \, dW_s
  =
  \int^t_0
  e^{(t-s)A}
  \big[
  P \big( B( X^{0,x}_s ) \big)
  \big]
  \, dW_s
  =
  0
  .
\end{split}
\end{equation}
This shows that for all 
$ x \in H $, 
$ t \in [0,T] $
it holds that 
\begin{equation}
  \P\bigg(
  B( X^{0,x}_t )
  =
  \sqrt{
    1
    +
    \| P X^{0,x}_t \|^2_H
  } \,
  e_1
  =
  \sqrt{
    1
    +
    \| P e^{tA} x \|^2_H
  } \,
  e_1
  =
  B( e^{tA} x )
  \bigg)
  =1
  .
\end{equation}
This and~\eqref{eq:proof.SEE} yield that 
for all 
$ x \in H $, 
$ t \in [0,T] $
it holds that 
\begin{equation}
\label{eq:base.SEE.formula}
  \stochval{X^{0,x}_t-e^{tA} x}
  =
  \int^t_0
  e^{ (t-s)A }
  B( e^{sA} x )
  \, \diffns W_s
  .
\end{equation}
Next note that item~\eqref{item:B.formula} of Lemma~\ref{lem:derivative.formula} ensures that for all 
$ n \in \N $, 
$ \mathbf{v} \in H^n $, 
$ x \in H $
it holds that 
$ P \big(B^{(n)}(x)\mathbf{v} \big) = 0 $. 
This and~\eqref{eq:proof.SEE} imply that for all 
$ n \in \N $, 
$ \mathbf{u} = (u_0,u_1,\ldots,u_n) \in H^{n+1} $, 
$ t \in [0,T] $ 
it holds that 
\begin{equation}
\label{eq:n.project}
\begin{split}
&
  \stochval{P X^{n,\mathbf{u}}_t
  	-
  	\1_{\{1\}}(n) \, P e^{tA} u_1}
\\&=
  P
  \int^t_0
  e^{(t-s)A}
  \sum_{ \varpi\in \Pi_n }
  B^{ ( \#_\varpi ) }( X_s^{ 0,u_0 } )
  \big(
  X_s^{ \#_{I^\varpi_1}, [ \mathbf{u} ]_1^{ \varpi } }
  ,
  X_s^{ \#_{I^\varpi_2}, [ \mathbf{u} ]_2^{ \varpi } }
  ,
  \dots
  ,
  X_s^{ \#_{I^\varpi_{\#_\varpi}}, [\mathbf{u} ]_{ \#_\varpi }^{ \varpi } }
  \big)
  \, dW_s
\\&=
  \int^t_0
  e^{(t-s)A}
  \sum_{ \varpi\in \Pi_n }
  \bigg[
  P\bigg( B^{ ( \#_\varpi ) }( X_s^{ 0,u_0 } )
  \big(
  X_s^{ \#_{I^\varpi_1}, [ \mathbf{u} ]_1^{ \varpi } }
  ,
  X_s^{ \#_{I^\varpi_2}, [ \mathbf{u} ]_2^{ \varpi } }
  ,
  \dots
  ,
  X_s^{ \#_{I^\varpi_{\#_\varpi}}, [\mathbf{u} ]_{ \#_\varpi }^{ \varpi } }
  \big)
  \bigg)
  \bigg]
  \, dW_s
  =0
  .
\end{split}
\end{equation}
Hence, we obtain that for all 
$n\in\{2,3,\ldots\}$, 
$\mathbf{u}\in H^{n+1}$, 
$t\in[0,T]$
it holds that 
\begin{equation}
\label{eq:derivative.zero.projection}
  \P\big(
    P(X^{n,\mathbf{u}}_t) = 0
  \big)
  = 1
  .
\end{equation}
In addition, note that item~\eqref{item:B.formula} of Lemma~\ref{lem:derivative.formula} implies that for all 
$n\in\N$,
$ v_0,v_1,\ldots,v_n \in H $ 
it holds that 
\begin{equation}
\label{eq:B.project.input}
B^{(n)}(v_0)(v_1,v_2,\ldots,v_n)=
B^{(n)}(Pv_0)(Pv_1,Pv_2,\ldots,Pv_n)
.
\end{equation}
Combining this with~\eqref{eq:derivative.zero.projection} ensures that for all 
$ n \in \N $, 
$ \varpi \in \Pi_n $, 
$ \mathbf{u} = (u_0,u_1,\ldots,u_n) \in H^{n+1} $, 
$ t \in [0,T] $
with 
$
  \varpi \neq \big\{ \{1\}, \{2\}, \ldots, \{n\} \big\}
$
it holds that 
\begin{equation}
\label{eq:0.B}
\begin{split}
&
      \P\bigg(
      B^{ ( \#_\varpi ) }( X_t^{ 0, u_0 } )
      \big(
        X_t^{ \#_{I^\varpi_1}, [ \mathbf{u} ]_1^{ \varpi } }
        ,
        X_t^{ \#_{I^\varpi_2}, [ \mathbf{u} ]_2^{ \varpi } }
        ,
        \dots
        ,
        X_t^{ \#_{I^\varpi_{\#_\varpi}}, [\mathbf{u} ]_{ \#_\varpi }^{ \varpi } }
      \big)
\\&\quad=
      B^{ ( \#_\varpi ) }( P X_t^{ 0, u_0 } )
      \big(
      P X_t^{ \#_{I^\varpi_1}, [ \mathbf{u} ]_1^{ \varpi } }
      ,
      P X_t^{ \#_{I^\varpi_2}, [ \mathbf{u} ]_2^{ \varpi } }
      ,
      \dots
      ,
      P X_t^{ \#_{I^\varpi_{\#_\varpi}}, [\mathbf{u} ]_{ \#_\varpi }^{ \varpi } }
      \big)
      =
      0
      \bigg)
      =1.
\end{split}
\end{equation}
Equation~\eqref{eq:B.project.input} hence implies that for all 
$ n \in \N $, 
$ \mathbf{u} = (u_0,u_1,\ldots,u_n) \in H^{n+1} $, 
$ t \in [0,T] $ 
it holds that 
\begin{equation}
\begin{split}
&
  \P\Bigg(
      \sum_{ \varpi \in \Pi_n }
      B^{ ( \#_\varpi ) }( X_t^{ 0, u_0 } )
      \big(
      X_t^{ \#_{I^\varpi_1}, [ \mathbf{u} ]_1^{ \varpi } }
      ,
      X_t^{ \#_{I^\varpi_2}, [ \mathbf{u} ]_2^{ \varpi } }
      ,
      \dots
      ,
      X_t^{ \#_{I^\varpi_{\#_\varpi}}, [\mathbf{u} ]_{ \#_\varpi }^{ \varpi } }
      \big)
\\&\quad=
      B^{ (n) }( X_t^{ 0, u_0 } )
      \big(
      X_t^{ 1, (u_0,u_1) }
      ,
      X_t^{ 1, (u_0,u_2) }
      ,
      \dots
      ,
      X_t^{ 1, (u_0,u_n) }
      \big)
\\&\quad=
      B^{ (n) }( P X_t^{ 0, u_0 } )
      \big(
      P X_t^{ 1, (u_0,u_1) }
      ,
      P X_t^{ 1, (u_0,u_2) }
      ,
      \dots
      ,
      P X_t^{ 1, (u_0,u_n) }
      \big)
  \Bigg)=1
.
\end{split}
\end{equation}
Combining this with~\eqref{eq:0.project}, \eqref{eq:n.project}, and~\eqref{eq:B.project.input} shows that for all 
$ n \in \N $, 
$ \mathbf{u} = (u_0,u_1,\ldots,u_n) \in H^{n+1} $, 
$ t \in [0,T] $ 
it holds that 
\begin{equation}
\begin{split}
&
\P\Bigg(
\sum_{ \varpi \in \Pi_n }
B^{ ( \#_\varpi ) }( X_t^{ 0, u_0 } )
\big(
X_t^{ \#_{I^\varpi_1}, [ \mathbf{u} ]_1^{ \varpi } }
,
X_t^{ \#_{I^\varpi_2}, [ \mathbf{u} ]_2^{ \varpi } }
,
\dots
,
X_t^{ \#_{I^\varpi_{\#_\varpi}}, [\mathbf{u} ]_{ \#_\varpi }^{ \varpi } }
\big)
\\&\quad=
B^{ ( n ) }( P e^{tA} { u_0 } )
\big(
P e^{tA} u_1
,
P e^{tA} u_2
,
\dots,
P e^{tA} u_n
\big)
\\&\quad=
B^{ ( n ) }( e^{tA} { u_0 } )
\big(
e^{tA} u_1
,
e^{tA} u_2
,
\dots,
e^{tA} u_n
\big)
\Bigg)=1
.
\end{split}
\end{equation}
This and~\eqref{eq:proof.SEE} assure that for all 
$ n \in \N $, 
$ \mathbf{u} = (u_0,u_1,\ldots,u_n) \in H^{n+1} $, 
$ t \in [0,T] $ 
it holds that 
\begin{equation}
\label{eq:n.SDE}
\begin{split}
&
  \stochval{X_t^{n,\mathbf{u}}
  	-
  	\mathbbm{1}_{ \{ 1 \} }(n) \, e^{tA} u_n }
=
  \int_0^t
  e^{ ( t - s ) A }
  B^{ ( n ) }( e^{sA} u_0 )
  \big(
  e^{sA} u_1
  ,
  e^{sA} u_2
  ,
  \dots
  ,
  e^{sA} u_n
  \big)
  \, \diffns W_s
  .
\end{split}
\end{equation}
Combining this and~\eqref{eq:base.SEE.formula} establishes item~\eqref{item:derivative.solution}.
The proof of Lemma~\ref{lem:derivative.process} is thus completed.
\end{proof}

\subsection{Disprove of regularities for the derivative processes}
\label{sec:disprove}

\begin{theorem}
\label{lem:blowup}
Assume the setting in Section~\ref{sec:counter.setting}, 
let $ c \in (0,\infty) $, 
and assume for all $ n \in \N $ that 
$ \lambda_n = -c n^2 $.
Then 
\begin{enumerate}[(i)]
\item
\label{item:thm.base.existence}
there exist up-to-modifications unique
$ ( \mathcal{F}_t )_{ t \in [0,T] } $/$ \mathcal{B}(H) $-predictable stochastic processes
$ X^{0,x} \colon 
$
$
[0,T] \times \Omega \to H $, 
$ x \in H $, 
which fulfill for all
$ p \in [2,\infty) $,
$ x \in H $,  
$ t \in [0,T] $
that
$
  \sup_{ s \in [0,T] }
  \E\big[\|
    X_s^{0,x}
  \|^p_H
  \big]
  < \infty
$
and
\begin{equation}
\begin{split}
&
  \stochval{X_t^{0,x}
  	-
  	e^{tA} x}
  =
  \int_0^t
    e^{ ( t - s ) A }
      B(X_s^{0,x})
  \, \diffns W_s
  ,
\end{split}
\end{equation}
\item
\label{item:thm.existence.derivatives}
it holds for all 
$ p \in [2,\infty) $, 
$ t \in [0,T] $ 
that 
$
  H \ni x \mapsto \stochval{X^{0,x}_t} \in \lpnb{p}{\P}{H}
$ 
is infinitely often Fr\'{e}chet differentiable,
\item
\label{item:thm.derivative.predictable}
there exist up-to-modifications unique
$ ( \mathcal{F}_t )_{ t \in [0,T] } $/$ \mathcal{B}(H) $-predictable stochastic processes
$
  X^{n,\mathbf{u}} \colon 
$
$
  [0,T] \times \Omega \to H
$, 
$
  \mathbf{u} \in H^{n+1}
$, 
$
  n \in \N
$,
which fulfill for all 
$ p \in [2,\infty) $, 
$
  n \in \N
$, 
$ \mathbf{u} \in H^n $, 
$
  x \in H
$, 
$ t \in [0,T] $ 
that  
\begin{equation}
\begin{split}
&
  \big(
  \tfrac{d^n}{dx^n}
  \stochval{X^{0,x}_t}
  \big) \mathbf{u}
  =
  \big(
    H \ni y \mapsto \stochval{X^{0,y}_t} \in \lpnb{p}{\P}{H}
  \big)^{(n)}(x) \, \mathbf{u}
  =
  \stochval{X^{n,(x,\mathbf{u})}_t}
  ,
\end{split}
\end{equation}
\item
\label{item:regularity}
it holds for all 
$ p \in [2,\infty) $, 
$ n \in \N $, 
$ q,\delta_1,\delta_2,\ldots,\delta_n \in [0,\infty) $, 
$ t \in (0,T] $ 
with 
$ \sum^n_{i=1} \delta_i < \nicefrac{1}{2} $
that
\begin{equation}
  \sup_{ \mathbf{u} = (u_0,u_1,\ldots,u_n) \in H \times (\nzspace{H})^n }
  \frac{
  	\big(\E\big[
  	\| X^{ n,\mathbf{u} }_t \|^p_{H_{-q}}
  	\big]\big)^{\nicefrac{1}{p}}
  }{
  \prod^{ n }_{ i=1 }
  \|u_i\|_{H_{-\delta_i}}
  }
  < \infty,
\end{equation}
and  
\item
\label{item:blowup}
it holds for all 
$ p \in [2,\infty) $, 
$ n \in \N $, 
$q\in[0,\infty)$, 
$ \delta_1, \delta_2, \ldots, \delta_n \in \R $, 
$ t \in (0,T] $ 
with 
$
  \sum^n_{i=1}
  \delta_i
  > \nicefrac{1}{2}
$
that 
\begin{equation}
  \sup_{ \mathbf{u}=(u_0,u_1,\ldots,u_n) \in (\nzspace{(\cap_{r\in\R}H_r)})^{n+1} }
  \frac{
    \big(\E\big[
    \| X^{ n,\mathbf{u} }_t \|^p_{H_{-q}}
    \big]\big)^{\nicefrac{1}{p}}
  }{
    \prod^{ n }_{ i=1 }
    \|u_i\|_{H_{-\delta_i}}
  }
  =
  \infty
  .
\end{equation}
\end{enumerate}
\end{theorem}
\begin{proof}
Throughout this proof 
let 
$
  v^{ k, r }_{ n,N }
  \in H
$, 
$ N,n \in \N $, 
$ k \in \N_0 $, 
$ r \in \R $, 
be the vectors which satisfy for all 
$ N,n \in \N $, 
$ k \in \N_0 $, 
$ r \in \R $ 
that 
\begin{equation}
\label{eq:def.v}
  v^{ k, r }_{ n,N }
  =
  (-A)^r
  \left[
  \sum^N_{ j=1 }
  e_{k+jn}
  \right]
  =
  \sum^N_{ j=1 }
  [ c \, (k+jn)^2 ]^r \,
  e_{k+jn}
  =
  c^r
  \left[
  \sum^N_{j=1}
  (k+jn)^{2r} \,
  e_{k+jn}
  \right]
  ,
\end{equation} 
let 
$
  \mathbf{u}^{\varepsilon,m,\boldsymbol{\delta}}_{n,N}
  \in H^{n+1}
$, 
$ \boldsymbol{\delta} \in \R^n $, 
$ \varepsilon \in \R $,
$ m \in \N_0 $,  
$ N,n \in \N $, 
be the vectors which satisfy for all 
$ N,n \in \N $, 
$ m \in \N_0 $,  
$ \varepsilon, \delta_1,\delta_2,\ldots,\delta_{2n+1} \in \R $ 
that 
\begin{equation}
  \mathbf{u}^{\varepsilon,m,\delta_1}_{1,N}
  =
  (v^{1,-\varepsilon}_{N^m,N},
  N^m \, v^{1,\delta_1-(\nicefrac{1}{2})-\varepsilon}_{N^m,N}),
\end{equation} 
\begin{multline}
\label{eq:def.u.even}
  \mathbf{u}^{\varepsilon,m,(\delta_1,\delta_2,\ldots,\delta_{2n})}_{2n,N}
=
    ( e_1, v^{1,\delta_1-\varepsilon}_{nN^m,N}, v^{1,\delta_2-\varepsilon}_{nN^m,N}, 
    v^{2,\delta_3-\varepsilon}_{nN^m,N}, v^{2,\delta_4-\varepsilon}_{nN^m,N},
    \ldots, 
\\
    v^{n-1,\delta_{2n-3}-\varepsilon}_{nN^m,N}, 
    v^{n-1,\delta_{2n-2}-\varepsilon}_{nN^m,N}, 
    v^{n,\delta_{2n-1}-\varepsilon}_{nN^m,N}, 
    N^m \, v^{n,\delta_{2n}-(\nicefrac{1}{2})-\varepsilon}_{nN^m,N} ),
\end{multline}
and 
\begin{multline}
\label{eq:def.u.odd}
\!\!\!
  \mathbf{u}^{\varepsilon,m,(\delta_1,\delta_2,\ldots,\delta_{2n+1})}_{2n+1,N}
=
  ( v^{n+1,-\varepsilon}_{(n+1)N^m,N}, v^{1,\delta_1-\varepsilon}_{(n+1)N^m,N}, v^{1,\delta_2-\varepsilon}_{(n+1)N^m,N}, 
  v^{2,\delta_3-\varepsilon}_{(n+1)N^m,N}, 
  v^{2,\delta_4-\varepsilon}_{(n+1)N^m,N}, 
  \ldots, 
\\
  v^{n,\delta_{2n-1}-\varepsilon}_{(n+1)N^m,N}, v^{n,\delta_{2n}-\varepsilon}_{(n+1)N^m,N}, 
  N^m \, v^{n+1,\delta_{2n+1}-(\nicefrac{1}{2})-\varepsilon}_{(n+1)N^m,N} ),
\end{multline}
let 
$
  \theta^n_i
  \colon H^n \to H
$, 
$ i \in \{1,2,\ldots,n\} $, 
$ n \in \N $, 
be the functions which satisfy for all 
$ n \in \N $, 
$ i \in \{1,2,\ldots,n\} $, 
$ \mathbf{u}=(u_1,u_2,\ldots,u_n) \in H^n $ 
that 
$
  \theta^n_i(\mathbf{u})
  = u_i
$, 
and let 
$ \lfloor \cdot \rfloor \colon \R \to \R $ 
and 
$ \lceil \cdot \rceil \colon \R \to \R $ 
be the functions which satisfy for all $ t \in \R $ that 
\begin{equation}
  \lfloor t \rfloor =
  \max\!\left(
  (-\infty,t]
  \cap
  \{ 0, 1 , - 1 , 2 , - 2 , \dots \}
  \right)
  =
  \max( (-\infty,t] \cap \mathbb{Z} )
\end{equation}
and
\begin{equation}
\begin{split}
  \lceil t \rceil =
  \min\!\left(
    [ t, \infty )
    \cap
    \{ 0, 1 , - 1 , 2 , - 2 , \dots \}
  \right)
  =
  \min( [t,\infty) \cap \mathbb{Z} )
  .
\end{split}
\end{equation}
Note that for all
$N,n\in\N$,
$m\in\N_0$,
$\varepsilon,\delta_1,\delta_2,\ldots,\delta_{2n-1}\in\R$
it holds that 
\begin{multline}
\label{eq:def.u.odd.combined}
\!\!\!
  \mathbf{u}^{\varepsilon,m,(\delta_1,\delta_2,\ldots,\delta_{2n-1})}_{2n-1,N}
=
  ( v^{n,-\varepsilon}_{nN^m,N}, v^{1,\delta_1-\varepsilon}_{nN^m,N}, v^{1,\delta_2-\varepsilon}_{nN^m,N}, 
  v^{2,\delta_3-\varepsilon}_{nN^m,N}, 
  v^{2,\delta_4-\varepsilon}_{nN^m,N}, 
  \ldots, 
\\
  v^{n-1,\delta_{2n-3}-\varepsilon}_{nN^m,N}, v^{n-1,\delta_{2n-2}-\varepsilon}_{nN^m,N}, 
  N^m \, v^{n,\delta_{2n-1}-(\nicefrac{1}{2})-\varepsilon}_{nN^m,N} )
  .
\end{multline}
Moreover, observe that items~\eqref{item:base.existence}--\eqref{item:lem.regularity} of Lemma~\ref{lem:derivative.process} establish 
items~\eqref{item:thm.base.existence}--\eqref{item:regularity}. 
It thus remains to prove item~\eqref{item:blowup}.
For this let $X^{n,\mathbf{u}}\colon[0,T]\times\Omega\to H$, 
$\mathbf{u}\in H^{n+1}$, 
$n\in\N_0$, 
be $(\mathcal{F}_t)_{t\in[0,T]}$/$\mathcal{B}(H)$-predictable stochastic processes which fulfill that for all 
$p\in[2,\infty)$, 
$n\in\N$,
$\mathbf{u}\in H^n$, 
$x\in H$, 
$t\in[0,T]$ 
it holds
\begin{enumerate}[(I)]
	\item 
	that 
	$
	\sup_{s\in[0,T]}
	\E\big[\|X^{0,x}_s\|^p_H\big]
	< \infty
	$, 
	\item
	that 
\begin{equation}
\begin{split}
&
\stochval{X_t^{0,x}
	-
	e^{tA} x}
=
\int_0^t
e^{ ( t - s ) A }
B(X_s^{0,x})
\, \diffns W_s
,
\end{split}
\end{equation}
and
\item
that
\begin{equation}
\begin{split}
&
\big(
\tfrac{d^n}{dx^n}
\stochval{X^{0,x}_t}
\big) \mathbf{u}
=
\big(
H \ni y \mapsto \stochval{X^{0,y}_t} \in \lpnb{p}{\P}{H}
\big)^{(n)}(x) \, \mathbf{u}
=
\stochval{X^{n,(x,\mathbf{u})}_t}
.
\end{split}
\end{equation}
\end{enumerate}
Next observe that the fact that 
$
  \forall \, n, j, k, l_1, l_2, \ldots, l_n \in \N, \,
  i \in \{1,2,\ldots,n\}
  \colon
  l_i + jk \geq 2
$
implies that for all 
$ N, n, k, l_1, l_2, \ldots, l_n \in \N $, 
$ r_1, r_2, \ldots, r_n \in \R $, 
$ t \in [0,T] $ 
it holds that 
\begin{equation}
\prod^n_{i=1}
\big\|
P e^{tA} v^{l_i,r_i}_{k,N}
\big\|^2_H
=
\prod^n_{i=1}
\big\|
P e^{tA}(-A)^{r_i} {\smallsum^N_{j=1}} e_{l_i+jk}
\big\|^2_H
=
\prod^n_{i=1}
\big\|
e^{tA}(-A)^{r_i} {\smallsum^N_{j=1}} e_{l_i+jk}
\big\|^2_H.
\end{equation}
This shows that for all 
$ N, n, k, l_1, l_2, \ldots, l_n \in \N $, 
$ r_1, r_2, \ldots, r_n \in \R $, 
$ t \in [0,T] $ 
it holds that 
\begin{equation}
\begin{split}
&
  \prod^n_{i=1}
  \big\|
    P e^{tA} v^{l_i,r_i}_{k,N}
  \big\|^2_H
=
  \prod^n_{i=1}
  \big\|
  {\smallsum^N_{j=1}} \big(e^{tA}(-A)^{r_i} e_{l_i+jk}\big)
  \big\|^2_H  
\\&=
  \prod^n_{i=1}
  \left(
  \sum^N_{j=1}
  \big\|
  e^{tA}(-A)^{r_i} e_{l_i+jk}
  \big\|^2_H
  \right)
=
  \prod^n_{i=1}
  \left(
  \sum^N_{j=1}
  \left[
    e^{-c(l_i+jk)^2 t} \,
    \|(-A)^{r_i} e_{l_i+jk}\|_H
  \right]^2
  \right)
\\&=
  \sum^N_{j_1=1}
  \sum^N_{j_2=1}
  \ldots
  \sum^N_{j_n=1}
  \left(
  \prod^n_{i=1}
  \left[
    e^{-c(l_i+j_i k)^2 t} \,
    \|(-A)^{r_i} e_{l_i+j_i k}\|_H
  \right]^2
  \right)
\\&\geq
  \sum^N_{j_1=1}
  \sum^{j_1}_{j_2=j_1}
  \sum^{j_1}_{j_3=j_1}
  \ldots
  \sum^{j_1}_{j_n=j_1}
  \left(
  \prod^n_{i=1}
  \left[
    e^{-c(l_i+j_i k)^2 t} \,
    \|(-A)^{r_i} e_{l_i+j_i k}\|_H
  \right]^2
  \right)
\\&=
  \sum^N_{j=1}
  \left(
  \prod^n_{i=1}
  \left[
    e^{-c(l_i+jk)^2 t} \,
    \|(-A)^{r_i} e_{l_i+jk}\|_H
  \right]^2
  \right)
=
  \sum^N_{j=1}
  \left(
  \prod^n_{i=1}
  \left[
    e^{-c(l_i+jk)^2 t} \,
    [c\,(l_i+jk)^2]^{r_i}
  \right]^2
  \right)
  .
\end{split}
\end{equation}
This and the fact that 
\begin{equation}
  \forall \, x \in \R
  \colon
  x=\max\{x,0\}+\min\{x,0\}=\max\{x,0\}-\max\{-x,0\}
\end{equation}
ensure that for all 
$ N, n, k, l_1, l_2, \ldots, l_n \in \N $, 
$ r_1, r_2, \ldots, r_n \in \R $, 
$ t \in [0,T] $ 
it holds that 
\begin{equation}
\begin{split}
  \prod^n_{i=1}
  \big\|
    P e^{tA} v^{l_i,r_i}_{k,N}
  \big\|^2_H
&\geq
  \sum^N_{j=1}
  \left(
  \prod^n_{i=1}
  \Bigg[
  \big|
    e^{-c(l_i+jk)^2 t} \,
    [c\,(\nicefrac{jk}{n})^2]^{r_i}
  \big|^2
  \bigg[
  \frac{l_i+jk}{(\nicefrac{jk}{n})}
  \bigg]^{4r_i}
  \Bigg]
  \right)
\\&=
  \sum^N_{j=1}
\left(
\prod^n_{i=1}
\Bigg[
\big|
e^{-c(l_i+jk)^2 t} \,
[c\,(\nicefrac{jk}{n})^2]^{r_i}
\big|^2
\bigg[
  n+\frac{nl_i}{jk}
\bigg]^{4r_i}
\Bigg]
\right)
\\&=
  \sum^N_{j=1}
  \left(
  \prod^n_{i=1}
  \Bigg[
  \big|
    e^{-c(l_i+jk)^2 t} \,
    [c\,(\nicefrac{jk}{n})^2]^{r_i}
  \big|^2
  \frac{
  (
  n+\nicefrac{nl_i}{(jk)}
  )^{4\max\{r_i,0\}}
  }{
  (
  n+\nicefrac{nl_i}{(jk)}
  )^{4\max\{-r_i,0\}}
  }
  \Bigg]
  \right)
\\&\geq
  \sum^N_{j=1}
  \left(
  \prod^n_{i=1}
    \Bigg[
  \frac{
  \big|
    e^{-c(l_i+jk)^2 t} \,
    [c\,(\nicefrac{jk}{n})^2]^{r_i}
  \big|^2
  }{
  (
  n+\nicefrac{nl_i}{(jk)}
  )^{4\max\{-r_i,0\}}
  }
  \Bigg]
  \right)
\\&\geq
  \sum^N_{j=1}
  \left(
  \prod^n_{i=1}
  \Bigg[
  \frac{
  \big|
    e^{-c(l_i+jk)^2 t} \,
    [c\,(\nicefrac{jk}{n})^2]^{r_i}
  \big|^2
  }{
  (
  n+n\max_{m\in\{1,2,\ldots,n\}} l_m
  )^{4\max\{-r_i,0\}}
  }
  \Bigg]
  \right)
  .
\end{split}
\end{equation}
This assures that for all 
$ N, n, l_1, l_2, \ldots, l_n \in \N $, 
$ k \in \{n,2n,3n,\ldots\} $, 
$ r_1, r_2, \ldots, r_n \in \R $, 
$ t \in [0,T] $ 
it holds that 
\begin{equation}
\begin{split}
&
  \prod^n_{i=1}
  \big\|
    P e^{tA} v^{l_i,r_i}_{k,N}
  \big\|^2_H
\geq
  \sum^N_{j=1}
  \left(
  \prod^n_{i=1}
  \Bigg[
  \frac{
  \big|
    e^{-c(l_i+jk)^2 t} \,
    [c\,(\nicefrac{jk}{n})^2]^{r_i}
  \big|^2
  }{
  (
  n+n\max_{m\in\{1,2,\ldots,n\}} l_m
  )^{4|r_i|}
  }
  \Bigg]
  \right)
\\&=
  \frac{1}{
  (
  n+n\max_{i\in\{1,2,\ldots,n\}} l_i
  )^{4\sum^n_{i=1}|r_i|}
  }
  \left[
  \sum^N_{j=1}
  \left(
  \prod^n_{i=1}
  \big|
    e^{-c(l_i+jk)^2 t} \,
    [c\,(\nicefrac{jk}{n})^2]^{r_i}
  \big|^2
  \right)
  \right]
\\&\geq
  \frac{1}{
  (
  n+n\max_{i\in\{1,2,\ldots,n\}} l_i
  )^{4\sum^n_{i=1}|r_i|}
  }
  \left[
  \sum^N_{j=1}
  \left(
  \prod^n_{i=1}
  \big|
    e^{-2c(l_i)^2 t} \,
    e^{-2c(jk)^2 t} \,
    [c\,(\nicefrac{jk}{n})^2]^{r_i}
  \big|^2
  \right)
  \right]
\\&=
  \frac{
  e^{-4ct\sum^n_{i=1} |l_i|^2}
  }{
  (
  n+n\max_{i\in\{1,2,\ldots,n\}} l_i
  )^{4\sum^n_{i=1}|r_i|}
  }
  \left[
  \sum^N_{j=1}
  \big|
    e^{-2c(jk)^2 n t} \,
      [c\,(\nicefrac{jk}{n})^2]^{\sum^n_{i=1} r_i}
  \big|^2
  \right]
  .
\end{split}
\end{equation}
Therefore, we obtain that for all 
$ N, n, l_1, l_2, \ldots, l_n \in \N $, 
$ k \in \{n,2n,3n,\ldots\} $, 
$ r_1, r_2, \ldots, r_n \in \R $, 
$ t \in [0,T] $ 
it holds that 
\begin{equation}
\label{eq:power.multiplication}
\begin{split}
  \prod^n_{i=1}
  \big\|
    P e^{tA} v^{l_i,r_i}_{k,N}
  \big\|^2_H
&\geq
  \frac{
  e^{-4ctn\max_{i\in\{1,2,\ldots,n\}} |l_i|^2}
  }{
  (
  n+n\max_{i\in\{1,2,\ldots,n\}} l_i
  )^{4\sum^n_{i=1}|r_i|}
  }
  \left[
  \sum^N_{j=1}
  \big|
    e^{-2n^3tc(\nicefrac{jk}{n})^2} \,
      [c\,(\nicefrac{jk}{n})^2]^{\sum^n_{i=1} r_i}
  \big|^2
  \right]
\\&=
  \frac{
	e^{-4ctn\max_{i\in\{1,2,\ldots,n\}} |l_i|^2}
}{
	(
	n+n\max_{i\in\{1,2,\ldots,n\}} l_i
	)^{4\sum^n_{i=1}|r_i|}
}
\left[
\sum^N_{j=1}
\|
e^{2n^3tA} (-A)^{(\sum^n_{i=1}r_i)} e_{(\nicefrac{jk}{n})}
\|^2_H
\right]
\\&=
  \frac{
  e^{-4ctn\max_{i\in\{1,2,\ldots,n\}} |l_i|^2}
  }{
  (
  n+n\max_{i\in\{1,2,\ldots,n\}} l_i
  )^{4\sum^n_{i=1}|r_i|}
  } \,
  \| e^{2n^3tA} (-A)^{(\sum^n_{i=1}r_i)} 
  {\smallsum^N_{j=1}} e_{(\nicefrac{jk}{n})} \|^2_H
\\&=
  \frac{
  e^{-4ctn\max_{i\in\{1,2,\ldots,n\}} |l_i|^2}
  }{
  (
  n+n\max_{i\in\{1,2,\ldots,n\}} l_i
  )^{4\sum^n_{i=1}|r_i|}
  } \,
  \|e^{2n^3tA} \, v^{0,\sum^n_{i=1}r_i}_{\nicefrac{k}{n},N}\|^2_H
  .
\end{split}
\end{equation}
Furthermore, note that for all 
$ N,n \in \N $, 
$ k_1, k_2 \in \{1,2,\ldots,n\} $, 
$ r_1,r_2 \in \R $, 
$ t \in [0,T] $ 
it holds that 
\begin{equation}
\label{eq:self.adjoint}
\begin{split}
&
  \langle P e^{tA} v^{k_1,r_1}_{n,N}, e^{tA} v^{k_2,r_2}_{n,N} \rangle_H
=
  \langle P e^{tA} (-A)^{r_1} v^{k_1,0}_{n,N}, P e^{tA} (-A)^{r_2} v^{k_2,0}_{n,N} \rangle_H
\\&=
  \1_{\{k_1\}}(k_2) \,
  \| P e^{tA} (-A)^{\nicefrac{(r_1+r_2)}{2}}
  v^{k_1,0}_{n,N} \|^2_H
  =
  \1_{\{k_1\}}(k_2) \,
  \| P e^{tA} v^{k_1,\nicefrac{(r_1+r_2)}{2}}_{n,N} \|^2_H
  .
\end{split}
\end{equation}
In particular, this implies that for all 
$ N,n \in \N $, 
$ k_1, k_2 \in \{1,2,\ldots,n\} $, 
$ r_1,r_2 \in \R $, 
$ t \in [0,T] $ 
with
$
  k_1\neq k_2
$
it holds that 
\begin{equation}
\label{eq:self.adjoint.0}
\begin{split}
\langle P e^{tA} v^{k_1,r_1}_{n,N}, e^{tA} v^{k_2,r_2}_{n,N} \rangle_H
=
0
.
\end{split}
\end{equation}
Next observe that items~\eqref{item:B.smoothness} and~\eqref{item:B.derivative.bounded} of Lemma~\ref{lem:derivative.formula} ensure that for all 
$ n \in \N $, 
$r\in[0,\infty)$, 
$
  u_0, u_1, \ldots, u_n \in H
$, 
$
  t\in[0,T]
$ 
it holds that 
\begin{equation}
  \int^t_0
  \|
  e^{(t-s)A }
  B^{(n)}( e^{sA} u_0 ) 
  ( e^{sA} u_1, e^{sA} u_2, \ldots, e^{sA} u_n ) 
  \|^2_{H_{-r}}
  \, ds
  < \infty
  .
\end{equation}
Item~\eqref{item:derivative.solution} of Lemma~\ref{lem:derivative.process} and 
It\^{o}'s isometry hence
show that for all 
$ n \in \N $, 
$r\in[0,\infty)$, 
$
  \mathbf{u}=
  (u_0, u_1, \ldots, u_n) \in H^{n+1}
$, 
$ t\in [0,T] $ 
it holds that 
\begin{equation}
\label{eq:delete.initial}
\begin{split}
&
  \E\big[ \| X^{ n,\mathbf{u} }_t \|^2_{H_{-r}} \big]
\\&=
  \mathbbm{1}_{\{1\}}(n)
  \Bigg(
  \| e^{tA} u_1 \|^2_{H_{-r}}
  +
  2 \,
  \E\!\left[
  \bigg<
    e^{tA} u_1
    ,
    \int^t_0
    e^{(t-s)A }
    B'( e^{sA} u_0 ) e^{sA} u_1 
    \, \diffns{W_s}
  \bigg>_{\!H_{-r}}
  \right]
  \Bigg)
\\&\quad+
  \E\!\left[\left\|
  \int^t_0
  (-A)^{-r}
  e^{(t-s)A }
  B^{(n)}( e^{sA} u_0 ) 
  ( e^{sA} u_1, e^{sA} u_2, \ldots, e^{sA} u_n ) 
  \,dW_s
    \right\|^2_H\right]
\\&=
  \mathbbm{1}_{\{1\}}(n)
  \bigg(
  \| e^{tA} u_1 \|^2_{H_{-r}}
  +
  2 \,
  \left<
    e^{tA} u_1
    ,
    \E\!\left[
    \int^t_0
    e^{(t-s)A }
    B'( e^{sA} u_0 ) e^{sA} u_1 
    \, \diffns{W_s}
    \right]
  \right>_{\!H_{-r}}
  \bigg)
\\&\quad+
  \int^t_0
    \|
  (-A)^{-r}
  e^{(t-s)A }
  B^{(n)}( e^{sA} u_0 ) 
  ( e^{sA} u_1, e^{sA} u_2, \ldots, e^{sA} u_n ) 
    \|^2_H
  \,\diffns{s}
\\&=
  \mathbbm{1}_{\{1\}}(n) \,
  \| e^{tA} u_1 \|^2_{H_{-r}}
\\&\quad+
  \int^t_0
    \|
  (-A)^{-r}
  e^{(t-s)A }
  B^{(n)}( e^{sA} \theta^{n+1}_{1}(\mathbf{u}) ) 
  ( e^{sA} \theta^{n+1}_{2}(\mathbf{u}), e^{sA} \theta^{n+1}_{3}(\mathbf{u}), \ldots, e^{sA} \theta^{n+1}_{n+1}(\mathbf{u}) ) 
    \|^2_H
  \,\diffns{s}
\\&\geq
  \int^t_0
    \|
  e^{(t-s)A }
 B^{(n)}( e^{sA} \theta^{n+1}_{1}(\mathbf{u}) ) 
( e^{sA} \theta^{n+1}_{2}(\mathbf{u}), e^{sA} \theta^{n+1}_{3}(\mathbf{u}), \ldots, e^{sA} \theta^{n+1}_{n+1}(\mathbf{u}) ) 
    \|^2_{H_{-r}}
  \,\diffns{s}
  .
\end{split}
\end{equation}
In particular, this shows that for all 
$ N,n \in \N $,
$ m \in \N_0 $, 
$r\in[0,\infty)$, 
$ \varepsilon,\delta_1,\delta_2,\ldots,\delta_{2n} \in \R $, 
$ t \in [0,T] $
it holds that 
\begin{multline}
\label{eq:delete.initial.odd}
  \E\Big[\big\|
    X^{ 2n-1,\mathbf{u}^{\varepsilon,m,(\delta_1,\delta_2,\ldots,\delta_{2n-1})}_{2n-1,N} }_t
  \big\|^2_{H_{-r}}\Big]
\\\geq
  \int^t_0
    \big\|
  e^{(t-s)A }
  B^{(2n-1)}( e^{sA} \theta^{2n}_1(\mathbf{u}^{\varepsilon,m,(\delta_1,\delta_2,\ldots,\delta_{2n-1})}_{2n-1,N}) )
  \big(
  e^{sA} \theta^{2n}_2(\mathbf{u}^{\varepsilon,m,(\delta_1,\delta_2,\ldots,\delta_{2n-1})}_{2n-1,N}),
\\
  e^{sA} \theta^{2n}_3(\mathbf{u}^{\varepsilon,m,(\delta_1,\delta_2,\ldots,\delta_{2n-1})}_{2n-1,N}),
  \ldots,
  e^{sA}\theta^{2n}_{2n}(\mathbf{u}^{\varepsilon,m,(\delta_1,\delta_2,\ldots,\delta_{2n-1})}_{2n-1,N})
  \big)
    \big\|^2_{H_{-r}}
  \,\diffns{s}
\end{multline}
and
\begin{multline}
\label{eq:delete.initial.even}
  \E\Big[\big\|
    X^{ 2n,\mathbf{u}^{\varepsilon,m,(\delta_1,\delta_2,\ldots,\delta_{2n})}_{2n,N} }_t
  \big\|^2_{H_{-r}}\Big]
\\\geq
  \int^t_0
    \big\|
  e^{(t-s)A }
  B^{(2n)}( e^{sA} \theta^{2n+1}_1(\mathbf{u}^{\varepsilon,m,(\delta_1,\delta_2,\ldots,\delta_{2n})}_{2n,N}) )
  \big(
  e^{sA} \theta^{2n+1}_2(\mathbf{u}^{\varepsilon,m,(\delta_1,\delta_2,\ldots,\delta_{2n})}_{2n,N}),
\\
  e^{sA} \theta^{2n+1}_3(\mathbf{u}^{\varepsilon,m,(\delta_1,\delta_2,\ldots,\delta_{2n})}_{2n,N}),
  \ldots,
  e^{sA}\theta^{2n+1}_{2n+1}(\mathbf{u}^{\varepsilon,m,(\delta_1,\delta_2,\ldots,\delta_{2n})}_{2n,N})
  \big)
    \big\|^2_{H_{-r}}
  \,\diffns{s}
  .
\end{multline}
In the next step we estimate the right hand sides of~\eqref{eq:delete.initial.odd} and~\eqref{eq:delete.initial.even} from below to establish suitable lower bounds for the left hand sides of~\eqref{eq:delete.initial.odd} and~\eqref{eq:delete.initial.even}, respectively.
We start with estimating the right hand side of~\eqref{eq:delete.initial.odd} from below.
Observe that~\eqref{eq:def.u.odd.combined} implies that for all 
$ N, n \in \N $,
$ m \in \N_0 $, 
$ \varepsilon, \delta_1,\delta_2,
\ldots,\delta_{2n-1} \in \R $, 
$ t \in [0,T] $
it holds that 
\begin{equation}
\begin{split}
&
  B^{(2n-1)}( e^{tA} \theta^{2n}_1(\mathbf{u}^{\varepsilon,m,(\delta_1,\delta_2,\ldots,\delta_{2n-1})}_{2n-1,N}) )
  \big(
  e^{tA} \theta^{2n}_2(\mathbf{u}^{\varepsilon,m,(\delta_1,\delta_2,\ldots,\delta_{2n-1})}_{2n-1,N}),
  e^{tA} \theta^{2n}_3(\mathbf{u}^{\varepsilon,m,(\delta_1,\delta_2,\ldots,\delta_{2n-1})}_{2n-1,N}),
  \ldots,
  \\&\quad
  e^{tA}\theta^{2n}_{2n}(\mathbf{u}^{\varepsilon,m,(\delta_1,\delta_2,\ldots,\delta_{2n-1})}_{2n-1,N})
  \big)
\\&=
  B^{(2n-1)}( e^{tA} v^{n,-\varepsilon}_{nN^m,N} )
  \big(
    e^{tA} v^{1,\delta_1-\varepsilon}_{nN^m,N},
    e^{tA} v^{1,\delta_2-\varepsilon}_{nN^m,N},
    e^{tA} v^{2,\delta_3-\varepsilon}_{nN^m,N},
    e^{tA} v^{2,\delta_4-\varepsilon}_{nN^m,N},
\\&\quad
    \ldots,
    e^{tA} v^{n-1,\delta_{2n-3}-\varepsilon}_{nN^m,N},
    e^{tA} v^{n-1,\delta_{2n-2}-\varepsilon}_{nN^m,N},
    N^m \, e^{tA} v^{n,\delta_{2n-1}-(\nicefrac{1}{2})-\varepsilon}_{nN^m,N}
  \big)
  .
\end{split}
\end{equation}
Item~\eqref{item:B.formula} of Lemma~\ref{lem:derivative.formula} therefore yields that for all 
$ N, n \in \N $,
$ m \in \N_0 $, 
$ \varepsilon, \delta_1,\delta_2,
\ldots,\delta_{2n-1} \in \R $, 
$ t \in [0,T] $
it holds that 
\begin{equation}
\label{eq:odd.formula}
\begin{split}
&
  B^{(2n-1)}( e^{tA} \theta^{2n}_1(\mathbf{u}^{\varepsilon,m,(\delta_1,\delta_2,\ldots,\delta_{2n-1})}_{2n-1,N}) )
  \big(
  e^{tA} \theta^{2n}_2(\mathbf{u}^{\varepsilon,m,(\delta_1,\delta_2,\ldots,\delta_{2n-1})}_{2n-1,N}),
  e^{tA} \theta^{2n}_3(\mathbf{u}^{\varepsilon,m,(\delta_1,\delta_2,\ldots,\delta_{2n-1})}_{2n-1,N}),
  \ldots,
  \\&\quad
  e^{tA}\theta^{2n}_{2n}(\mathbf{u}^{\varepsilon,m,(\delta_1,\delta_2,\ldots,\delta_{2n-1})}_{2n-1,N})
  \big)
\\&=
  \Bigg(
  {\sum\limits_{
    \varpi \in \Pi_{2n-1}
  }}
  \Bigg(
  \frac{
    \big(
      {\prod^{ \#_\varpi - 1 }_{ i=0 }}
      (1-2i)
    \big)
  }{
  \big[ 1 + \| P e^{tA} v^{n,-\varepsilon}_{nN^m,N} \|^2_H \big]^{
    ( \#_\varpi - \nicefrac{1}{2} )
  }
  }
\\&\cdot
  \Bigg[
  \prod_{ I \in \varpi }
  \bigg(
  \Big<
    \mathbbm{1}_{\{1\}}( \#_I ) \, P e^{tA} v^{n,-\varepsilon}_{nN^m,N} 
    ,
    N^{m\mathbbm{1}_{\{2n-1\}}(\min(I))}
    e^{tA} v^{\lceil \nicefrac{\min(I)}{2} \rceil,\delta_{\min(I)}-(\nicefrac{1}{2})\mathbbm{1}_{\{2n-1\}}(\min(I))-\varepsilon}_{nN^m,N}
  \Big>_H
\\&+
  \Big<
    \mathbbm{1}_{\{2\}}( \#_I ) \, P \big( N^{m\mathbbm{1}_{\{2n-1\}}(\max(I))} e^{tA} v^{\lceil \nicefrac{\max(I)}{2} \rceil,\delta_{\max(I)}-(\nicefrac{1}{2})\mathbbm{1}_{\{2n-1\}}(\max(I))-\varepsilon}_{nN^m,N}
    \big)
    ,
    e^{tA} v^{\lceil \nicefrac{\min(I)}{2} \rceil,\delta_{\min(I)}-\varepsilon}_{nN^m,N}
  \Big>_H
  \bigg)
  \Bigg]
  \Bigg)
  \Bigg)
  e_1
  .
\end{split}
\end{equation}
This and~\eqref{eq:self.adjoint.0}
imply that for all 
$ N, n \in \N $,
$ m \in \N_0 $, 
$ \varepsilon, \delta_1,\delta_2,
\ldots,\delta_{2n-1} \in \R $, 
$ t \in [0,T] $
it holds that 
\begin{equation}
\begin{split}
&
  B^{(2n-1)}( e^{tA} \theta^{2n}_1(\mathbf{u}^{\varepsilon,m,(\delta_1,\delta_2,\ldots,\delta_{2n-1})}_{2n-1,N}) )
  \big(
  e^{tA} \theta^{2n}_2(\mathbf{u}^{\varepsilon,m,(\delta_1,\delta_2,\ldots,\delta_{2n-1})}_{2n-1,N}),
  e^{tA} \theta^{2n}_3(\mathbf{u}^{\varepsilon,m,(\delta_1,\delta_2,\ldots,\delta_{2n-1})}_{2n-1,N}),
  \ldots,
  \\&\quad
  e^{tA}\theta^{2n}_{2n}(\mathbf{u}^{\varepsilon,m,(\delta_1,\delta_2,\ldots,\delta_{2n-1})}_{2n-1,N})
  \big)
\\&=
  \Bigg(
  \sum_{
    \varpi \in \Pi_{2n-1}, \,
    \varpi = \{ \{1,2\}, \{3,4\}, \ldots, \{2n-3,2n-2\}, \{2n-1\} \}
  }
  \Bigg(
  \frac{
      \big(\!
      \prod^{ \#_\varpi - 1 }_{ i=0 }
      (1-2i)
      \big)
  }{
    [ 1 + \| P e^{tA} v^{n,-\varepsilon}_{nN^m,N} \|^2_H ]^{
      ( \#_\varpi - 1/2 )
    }
  }
\\ & \quad \cdot
  \big<
    P e^{tA} v^{n,-\varepsilon}_{nN^m,N}
    ,
    N^m e^{tA} v^{n,\delta_{2n-1}-(\nicefrac{1}{2})-\varepsilon}_{nN^m,N}
  \big>_H
  \Bigg[
  \prod^{n-1}_{ i=1 }
  \big<
    P e^{tA} v^{i,\delta_{2i}-\varepsilon}_{nN^m,N} 
    ,
    e^{tA} v^{i,\delta_{2i-1}-\varepsilon}_{nN^m,N}
  \big>_H
  \Bigg]
  \Bigg)
  \Bigg)
  e_1
.
\end{split}
\end{equation}
Identity~\eqref{eq:self.adjoint}
therefore shows that for all 
$ N, n \in \N $,
$ m \in \N_0 $, 
$ \varepsilon, \delta_1,\delta_2,
\ldots,\delta_{2n-1} \in \R $, 
$ t \in [0,T] $
it holds that 
\begin{equation}
\label{eq:odd.simplification}
\begin{split}
&
B^{(2n-1)}( e^{tA} \theta^{2n}_1(\mathbf{u}^{\varepsilon,m,(\delta_1,\delta_2,\ldots,\delta_{2n-1})}_{2n-1,N}) )
\big(
e^{tA} \theta^{2n}_2(\mathbf{u}^{\varepsilon,m,(\delta_1,\delta_2,\ldots,\delta_{2n-1})}_{2n-1,N}),
e^{tA} \theta^{2n}_3(\mathbf{u}^{\varepsilon,m,(\delta_1,\delta_2,\ldots,\delta_{2n-1})}_{2n-1,N}),
\ldots,
\\&\quad
e^{tA}\theta^{2n}_{2n}(\mathbf{u}^{\varepsilon,m,(\delta_1,\delta_2,\ldots,\delta_{2n-1})}_{2n-1,N})
\big)
\\&=
N^m \, \Bigg(
\frac{
	\big(\!
	\prod^{n-1}_{ i=0 }
	(1-2i)
	\big)
}{
	[ 1 + \| P e^{tA} v^{n,-\varepsilon}_{nN^m,N} \|^2_H ]^{
		( n - 1/2 )
	}
}\,
\| P e^{tA} v^{n,(\nicefrac{\delta_{(2n-1)}}{2})-(\nicefrac{1}{4})-\varepsilon}_{nN^m,N} \|^2_H
\\&\quad\cdot
\Bigg[
\prod^{n-1}_{i=1}
\| P e^{tA} v^{i,(\nicefrac{(\delta_{2i-1}+\delta_{2i})}{2})-\varepsilon}_{nN^m,N} \|^2_H
\Bigg]
\Bigg)
e_1
.
\end{split}
\end{equation}
Hence, we obtain that for all 
$ N,n \in \N $,
$ m \in \N_0 $, 
$r\in[0,\infty)$, 
$ \delta_1,\delta_2,\ldots,\delta_{2n-1} \in \R $, 
$ \varepsilon \in (0,\infty) $, 
$ t \in [0,T] $, 
$ s \in [0,t] $
it holds that 
\begin{equation}
\begin{split}
&
    \big\|
  e^{(t-s)A }
  B^{(2n-1)}( e^{sA} \theta^{2n}_1(\mathbf{u}^{\varepsilon,m,(\delta_1,\delta_2,\ldots,\delta_{2n-1})}_{2n-1,N}) )
  \big(
  e^{sA} \theta^{2n}_2(\mathbf{u}^{\varepsilon,m,(\delta_1,\delta_2,\ldots,\delta_{2n-1})}_{2n-1,N}),
  \\&\quad
  e^{sA} \theta^{2n}_3(\mathbf{u}^{\varepsilon,m,(\delta_1,\delta_2,\ldots,\delta_{2n-1})}_{2n-1,N}),
  \ldots,
  e^{sA}\theta^{2n}_{2n}(\mathbf{u}^{\varepsilon,m,(\delta_1,\delta_2,\ldots,\delta_{2n-1})}_{2n-1,N})
  \big)
    \big\|^2_{H_{-r}}
\\&=
  \bigg(
  \frac{N^m}{c^r}
  \bigg)^2
  e^{ -2c(t-s) }
  \Bigg[
    \frac{
      \big(
      \prod\nolimits^{n-1}_{ i=0 }
      (1-2i)
      \big)
    }{
  [ 1 + \| P e^{sA} v^{n,-\varepsilon}_{nN^m,N} \|^2_H ]^{
    ( n - 1/2 )
  }
    } \,
  \| P e^{sA} v^{n,(\nicefrac{\delta_{(2n-1)}}{2})-(\nicefrac{1}{4})-\varepsilon}_{nN^m,N} \|^2_H
\\&\quad\cdot
  \Bigg(
  \prod^{n-1}_{i=1}
  \| P e^{sA} v^{i,(\nicefrac{(\delta_{2i-1}+\delta_{2i})}{2})-\varepsilon}_{nN^m,N} \|^2_H
  \Bigg)
  \Bigg]^2
  .
\end{split}
\end{equation}
Plugging this into the right hand side of~\eqref{eq:delete.initial.odd} yields that for all 
$ N,n \in \N $,
$ m \in \N_0 $, 
$r\in[0,\infty)$, 
$ \delta_1,\delta_2,\ldots,\delta_{2n-1} \in \R $, 
$ \varepsilon \in (0,\infty) $, 
$ t \in (0,T] $
it holds that 
\begin{equation}
\begin{split}
&
  \E\Big[\big\|
    X^{ 2n-1,\mathbf{u}^{\varepsilon,m,(\delta_1,\delta_2,\ldots,\delta_{2n-1})}_{2n-1,N} }_t
  \big\|^2_{H_{-r}}\Big]
\\&\geq
  \bigg(
  \frac{N^m}{c^r}
  \bigg)^2
  \int^t_0
  e^{ -2c(t-s) }
  \Bigg[
    \frac{
      \big(
      \prod\nolimits^{n-1}_{ i=0 }
      (1-2i)
      \big)
    }{
  [ 1 + \| P e^{sA} v^{n,-\varepsilon}_{nN^m,N} \|^2_H ]^{
    ( n - 1/2 )
  }
    } \,
  \| P e^{sA} v^{n,(\nicefrac{\delta_{(2n-1)}}{2})-(\nicefrac{1}{4})-\varepsilon}_{nN^m,N} \|^2_H
\\&\quad\cdot
  \Bigg(
  \prod^{n-1}_{i=1}
  \| P e^{sA} v^{i,(\nicefrac{(\delta_{2i-1}+\delta_{2i})}{2})-\varepsilon}_{nN^m,N} \|^2_H
  \Bigg)
  \Bigg]^2
  \,\diffns{s}
\\&\geq
  \left[
    \frac{
      N^m
      \big(
      \prod\nolimits^{n-1}_{ i=0 }
      (1-2i)
      \big)
    }{
  c^r e^{ct} [ 1 + \sup_{s\in[0,T]}\| P e^{sA} v^{n,-\varepsilon}_{nN^m,N} \|^2_H ]^{
    ( n - 1/2 )
  }
    }
  \right]^2
\\&\quad\cdot
  \int^t_0
  \Bigg[
  \| P e^{sA} v^{n,(\nicefrac{\delta_{(2n-1)}}{2})-(\nicefrac{1}{4})-\varepsilon}_{nN^m,N} \|^2_H
  \Bigg(
  \prod^{n-1}_{i=1}
  \| P e^{sA} v^{i,(\nicefrac{(\delta_{2i-1}+\delta_{2i})}{2})-\varepsilon}_{nN^m,N} \|^2_H
  \Bigg)
  \Bigg]^2
  \,\diffns{s}
  .
\end{split}
\end{equation}
Combining this with~\eqref{eq:power.multiplication} ensures that for all 
$ N,n \in \N $,
$ m \in \N_0 $, 
$r\in[0,\infty)$, 
$ \delta_1,\delta_2,\ldots,\delta_{2n-1} \in \R $, 
$ \varepsilon \in (0,\infty) $, 
$ t \in (0,T] $
it holds that 
\begin{equation}
\label{eq:lb.odd}
\begin{split}
&
  \E\Big[\big\|
    X^{ 2n-1,\mathbf{u}^{\varepsilon,m,(\delta_1,\delta_2,\ldots,\delta_{2n-1})}_{2n-1,N} }_t
  \big\|^2_{H_{-r}}\Big]
\\&\geq
  \left[
    \frac{
      N^m
      \big(
      \prod\nolimits^{n-1}_{ i=0 }
      (1-2i)
      \big)
    }{
  c^r e^{ct} [ 1 + \sup_{s\in[0,T]}\| P e^{sA} v^{n,-\varepsilon}_{nN^m,N} \|^2_H ]^{
    ( n - 1/2 )
  }
    }
  \right]^2
\\&\quad\cdot
  \int^t_0
  \frac{
    e^{ -8cn^3s }
    \| e^{2n^3sA} v^{0,-(\nicefrac{1}{4})-n\varepsilon+\sum^{2n-1}_{i=1}(\nicefrac{\delta_i}{2})}_{N^m,N} \|^4_H
  }{
    (n+n^2)^{8(|(\nicefrac{\delta_{(2n-1)}}{2})-(\nicefrac{1}{4})-\varepsilon|+\sum^{n-1}_{i=1}|(\nicefrac{(\delta_{2i-1}+\delta_{2i})}{2})-\varepsilon|)}
  }
  \,\diffns{s}
\\&\geq
  \left[
    \frac{
      N^m
      \big(
      \prod\nolimits^{n-1}_{ i=0 }
      (1-2i)
      \big)
    }{
  c^r
  e^{ c(1+4n^3) t}
    (2n^2)^{(1+4n\varepsilon+2\sum^{2n-1}_{i=1}|\delta_i|)} \,
  [ 1 + \sup_{s\in[0,T]} \| P e^{sA} v^{n,-\varepsilon}_{nN^m,N} \|^2_H ]^{
    ( n - 1/2 )
  }
    }
  \right]^2
\\&\quad\cdot
  \int^t_0
  \| e^{2n^3sA} v^{0,-(\nicefrac{1}{4})-n\varepsilon+\sum^{2n-1}_{i=1}(\nicefrac{\delta_i}{2})}_{N^m,N} \|^4_H
  \,\diffns{s}
  .
\end{split}
\end{equation}
Next we estimate the right hand side of~\eqref{eq:delete.initial.even} from below.
Note that~\eqref{eq:def.u.even} 
shows that for all 
$ N, n \in \N $,
$ m \in \N_0 $, 
$ \varepsilon, \delta_1,\delta_2,
\ldots,\delta_{2n} \in \R $, 
$ t \in [0,T] $
it holds that 
\begin{equation}
\begin{split}
&
  B^{(2n)}( e^{tA} \theta^{2n+1}_1(\mathbf{u}^{\varepsilon,m,(\delta_1,\delta_2,\ldots,\delta_{2n})}_{2n,N}) )
  \big(
  e^{tA} \theta^{2n+1}_2(\mathbf{u}^{\varepsilon,m,(\delta_1,\delta_2,\ldots,\delta_{2n})}_{2n,N}),
  e^{tA} \theta^{2n+1}_3(\mathbf{u}^{\varepsilon,m,(\delta_1,\delta_2,\ldots,\delta_{2n})}_{2n,N}),
  \ldots,
  \\&\quad
  e^{tA}\theta^{2n+1}_{2n+1}(\mathbf{u}^{\varepsilon,m,(\delta_1,\delta_2,\ldots,\delta_{2n})}_{2n,N})
  \big)
  \\&=
  B^{(2n)}( e^{tA} e_1 )
  \big(
    e^{tA} v^{1,\delta_1-\varepsilon}_{nN^m,N},
    e^{tA} v^{1,\delta_2-\varepsilon}_{nN^m,N},
    e^{tA} v^{2,\delta_3-\varepsilon}_{nN^m,N},
    e^{tA} v^{2,\delta_4-\varepsilon}_{nN^m,N},
\\&\quad    
    \ldots,
    e^{tA} v^{n-1,\delta_{2n-3}-\varepsilon}_{nN^m,N},
    e^{tA} v^{n-1,\delta_{2n-2}-\varepsilon}_{nN^m,N},
    e^{tA} v^{n,\delta_{2n-1}-\varepsilon}_{nN^m,N},
    N^m \, e^{tA} v^{n,\delta_{2n}-(\nicefrac{1}{2})-\varepsilon}_{nN^m,N}
  \big)
  .
\end{split}
\end{equation}
Item~\eqref{item:B.formula} of Lemma~\ref{lem:derivative.formula} therefore ensures that for all 
$ N, n \in \N $,
$ m \in \N_0 $, 
$ \varepsilon, \delta_1,\delta_2,
\ldots,\delta_{2n} \in \R $, 
$ t \in [0,T] $
it holds that 
\begin{equation}
\label{eq:even.formula}
\begin{split}
&
  B^{(2n)}( e^{tA} \theta^{2n+1}_1(\mathbf{u}^{\varepsilon,m,(\delta_1,\delta_2,\ldots,\delta_{2n})}_{2n,N}) )
  \big(
  e^{tA} \theta^{2n+1}_2(\mathbf{u}^{\varepsilon,m,(\delta_1,\delta_2,\ldots,\delta_{2n})}_{2n,N}),
  e^{tA} \theta^{2n+1}_3(\mathbf{u}^{\varepsilon,m,(\delta_1,\delta_2,\ldots,\delta_{2n})}_{2n,N}),
  \ldots,
  \\&\quad
  e^{tA}\theta^{2n+1}_{2n+1}(\mathbf{u}^{\varepsilon,m,(\delta_1,\delta_2,\ldots,\delta_{2n})}_{2n,N})
  \big)
\\&=
  \Bigg(
  {\sum\limits_{
    \varpi \in \Pi_{2n}
  }}
  \Bigg(
  \frac{
    \big(
      {\prod^{ \#_\varpi - 1 }_{ i=0 }}
      (1-2i)
    \big)
  }{
  \big[ 1 + \| P e^{tA} e_1 \|^2_H \big]^{
    ( \#_\varpi - \nicefrac{1}{2} )
  }
  }
\\&\cdot
  \Bigg[
  \prod_{ I \in \varpi }
  \bigg(
  \Big<
    \mathbbm{1}_{\{1\}}( \#_I ) \, P e^{tA} e_1 
    ,
    N^{m\mathbbm{1}_{\{2n\}}(\min(I))}
    e^{tA} v^{\lceil \nicefrac{\min(I)}{2} \rceil,\delta_{\min(I)}-(\nicefrac{1}{2})\mathbbm{1}_{\{2n\}}(\min(I))-\varepsilon}_{nN^m,N}
  \Big>_H
\\&+
  \Big<
    \mathbbm{1}_{\{2\}}( \#_I ) \, P \big( N^{m\mathbbm{1}_{\{2n\}}(\max(I))} e^{tA} v^{\lceil \nicefrac{\max(I)}{2} \rceil,\delta_{\max(I)}-(\nicefrac{1}{2})\mathbbm{1}_{\{2n\}}(\max(I))-\varepsilon}_{nN^m,N}
    \big)
    ,
    e^{tA} v^{\lceil \nicefrac{\min(I)}{2} \rceil,\delta_{\min(I)}-\varepsilon}_{nN^m,N}
  \Big>_H
  \bigg)
  \Bigg]
  \Bigg)
  \Bigg)
  e_1
\\&=
  \Bigg(
  {\sum\limits_{
    \substack{\varpi \in \Pi_{2n}, \, \forall \, I \in \varpi \colon \#_I = 2}
  }}
  \Bigg(
    \Bigg[
      {\prod^{ \#_\varpi - 1 }_{ i=0 }}
      (1-2i)
    \Bigg]
\\&\cdot
  \Bigg[
  \prod_{\substack{ I \in \varpi }}
  \Big<
    P \big( N^{m\mathbbm{1}_{\{2n\}}(\max(I))} e^{tA} v^{\lceil \nicefrac{\max(I)}{2} \rceil,\delta_{\max(I)}-(\nicefrac{1}{2})\mathbbm{1}_{\{2n\}}(\max(I))-\varepsilon}_{nN^m,N}
    \big)
    ,
    e^{tA} v^{\lceil \nicefrac{\min(I)}{2} \rceil,\delta_{\min(I)}-\varepsilon}_{nN^m,N}
  \Big>_H
  \Bigg]
  \Bigg)
  \Bigg)
  e_1
  .
\end{split}
\end{equation}
This and~\eqref{eq:self.adjoint.0}
assure that for all 
$ N, n \in \N $,
$ m \in \N_0 $, 
$ \varepsilon, \delta_1,\delta_2,
\ldots,\delta_{2n} \in \R $, 
$ t \in [0,T] $
it holds that 
\begin{equation}
\begin{split}
&
  B^{(2n)}( e^{tA} \theta^{2n+1}_1(\mathbf{u}^{\varepsilon,m,(\delta_1,\delta_2,\ldots,\delta_{2n})}_{2n,N}) )
  \big(
  e^{tA} \theta^{2n+1}_2(\mathbf{u}^{\varepsilon,m,(\delta_1,\delta_2,\ldots,\delta_{2n})}_{2n,N}),
  e^{tA} \theta^{2n+1}_3(\mathbf{u}^{\varepsilon,m,(\delta_1,\delta_2,\ldots,\delta_{2n})}_{2n,N}),
  \ldots,
  \\&\quad
  e^{tA}\theta^{2n+1}_{2n+1}(\mathbf{u}^{\varepsilon,m,(\delta_1,\delta_2,\ldots,\delta_{2n})}_{2n,N})
  \big)
\\&=
\Bigg(
\sum_{
	\varpi \in \Pi_{2n}, \,
	\varpi = \{ \{1,2\}, \{3,4\}, \ldots, \{2n-3,2n-2\}, \{2n-1,2n\} \}
}
	\Bigg(
	\Bigg[
	{\prod\limits^{ \#_\varpi - 1 }_{ i=0 }
	(1-2i)}
	\Bigg]
\\ & \quad \cdot
\big<
P\big( N^m e^{tA} v^{n,\delta_{2n}-(\nicefrac{1}{2})-\varepsilon}_{nN^m,N} \big)
,
e^{tA} v^{n,\delta_{2n-1}-\varepsilon}_{nN^m,N}
\big>_H
\Bigg[
\prod^{n-1}_{ i=1 }
\big<
P e^{tA} v^{i,\delta_{2i}-\varepsilon}_{nN^m,N}
,
e^{tA} v^{i,\delta_{2i-1}-\varepsilon}_{nN^m,N}
\big>_H
\Bigg]
\Bigg)
\Bigg) \,
e_1
.
\end{split}
\end{equation}
Furthermore, identity~\eqref{eq:self.adjoint}
implies that for all 
$ N, n \in \N $,
$ m \in \N_0 $, 
$ \varepsilon, \delta_1,\delta_2,
\ldots,\delta_{2n} \in \R $, 
$ t \in [0,T] $
it holds that 
\begin{equation}
\label{eq:even.simplification}
\begin{split}
&
B^{(2n)}( e^{tA} \theta^{2n+1}_1(\mathbf{u}^{\varepsilon,m,(\delta_1,\delta_2,\ldots,\delta_{2n})}_{2n,N}) )
\big(
e^{tA} \theta^{2n+1}_2(\mathbf{u}^{\varepsilon,m,(\delta_1,\delta_2,\ldots,\delta_{2n})}_{2n,N}),
e^{tA} \theta^{2n+1}_3(\mathbf{u}^{\varepsilon,m,(\delta_1,\delta_2,\ldots,\delta_{2n})}_{2n,N}),
\ldots,
\\&\quad
e^{tA}\theta^{2n+1}_{2n+1}(\mathbf{u}^{\varepsilon,m,(\delta_1,\delta_2,\ldots,\delta_{2n})}_{2n,N})
\big)
\\&=
N^m \, \Bigg(
\left[
\prod^{n-1}_{ i=0 }
(1-2i)
\right]
\| P e^{tA} v^{n,(\nicefrac{(\delta_{2n-1}+\delta_{2n})}{2})-(\nicefrac{1}{4})-\varepsilon}_{nN^m,N} \|^2_H
\Bigg[
\prod^{n-1}_{i=1}
\| P e^{tA} v^{i,(\nicefrac{(\delta_{2i-1}+\delta_{2i})}{2})-\varepsilon}_{nN^m,N} \|^2_H
\Bigg]
\Bigg) \,
e_1.
\end{split}
\end{equation}
We therefore obtain that for all
$ N,n \in \N $,
$ m \in \N_0 $, 
$r\in[0,\infty)$, 
$ \delta_1,\delta_2,\ldots,\delta_{2n} \in \R $, 
$ \varepsilon \in (0,\infty) $, 
$ t \in [0,T] $, 
$ s \in [0,t] $
it holds that 
\begin{equation}
\begin{split}
&
    \big\|
  e^{(t-s)A }
  B^{(2n)}( e^{sA} \theta^{2n+1}_1(\mathbf{u}^{\varepsilon,m,(\delta_1,\delta_2,\ldots,\delta_{2n})}_{2n,N}) )
  \big(
  e^{sA} \theta^{2n+1}_2(\mathbf{u}^{\varepsilon,m,(\delta_1,\delta_2,\ldots,\delta_{2n})}_{2n,N}),
\\&\quad
  e^{sA} \theta^{2n+1}_3(\mathbf{u}^{\varepsilon,m,(\delta_1,\delta_2,\ldots,\delta_{2n})}_{2n,N}),
  \ldots,
  e^{sA}\theta^{2n+1}_{2n+1}(\mathbf{u}^{\varepsilon,m,(\delta_1,\delta_2,\ldots,\delta_{2n})}_{2n,N})
  \big)
    \big\|^2_{H_{-r}}
\\&=
  \bigg[
  \frac{N^m}{c^r}
  \bigg]^2
  e^{ -2c(t-s) }
  \Bigg[
    \left(
      \prod^{n-1}_{ i=0 }
      (1-2i)
    \right)
  \| P e^{sA} v^{n,\nicefrac{((\delta_{2n-1}+\delta_{2n})}{2})-(\nicefrac{1}{4})-\varepsilon}_{nN^m,N} \|^2_H
\\&\quad\cdot
  \Bigg(
  \prod^{n-1}_{i=1}
  \| P e^{sA} v^{i,\nicefrac{((\delta_{2i-1}+\delta_{2i})}{2})-\varepsilon}_{nN^m,N} \|^2_H
  \Bigg)\Bigg]^2
  .
\end{split}
\end{equation}
Plugging this into the right hand side of~\eqref{eq:delete.initial.even} yields that for all 
$ N,n \in \N $,
$ m \in \N_0 $, 
$r\in[0,\infty)$, 
$ \delta_1,\delta_2,\ldots,\delta_{2n} \in \R $, 
$ \varepsilon \in (0,\infty) $, 
$ t \in [0,T] $
it holds that 
\begin{equation}
\begin{split}
&
  \E\Big[\big\|
    X^{ 2n,\mathbf{u}^{\varepsilon,m,(\delta_1,\delta_2,\ldots,\delta_{2n})}_{2n,N} }_t
  \big\|^2_{H_{-r}}\Big]
\\&\geq
  \bigg[
  \frac{N^m}{c^r}
  \bigg]^2
  \int^t_0
  e^{ -2c(t-s) }
  \Bigg[
    \left(
      \prod^{n-1}_{ i=0 }
      (1-2i)
    \right)
  \| P e^{sA} v^{n,\nicefrac{((\delta_{2n-1}+\delta_{2n})}{2})-(\nicefrac{1}{4})-\varepsilon}_{nN^m,N} \|^2_H
\\&\quad\cdot
  \Bigg(
  \prod^{n-1}_{i=1}
  \| P e^{sA} v^{i,\nicefrac{((\delta_{2i-1}+\delta_{2i})}{2})-\varepsilon}_{nN^m,N} \|^2_H
  \Bigg)\Bigg]^2
  \,\diffns{s}
\\&\geq
  \left[
  \frac{
      N^m
      \big(
      \prod\nolimits^{n-1}_{ i=0 }
      (1-2i)
      \big)
  }{
    c^r e^{ ct }
  }
  \right]^2
\\&\quad\cdot
  \int^t_0
  \left[
  \| P e^{sA} v^{n,\nicefrac{((\delta_{2n-1}+\delta_{2n})}{2})-(\nicefrac{1}{4})-\varepsilon}_{nN^m,N} \|^2_H
  \Bigg(
  \prod^{n-1}_{i=1}
  \| P e^{sA} v^{i,(\nicefrac{(\delta_{2i-1}+\delta_{2i})}{2})-\varepsilon}_{nN^m,N} \|^2_H
  \Bigg)
  \right]^2
  \,\diffns{s}
  .
\end{split}
\end{equation}
This and~\eqref{eq:power.multiplication} assure that for all 
$ N,n \in \N $,
$ m \in \N_0 $, 
$r\in[0,\infty)$, 
$ \delta_1,\delta_2,\ldots,\delta_{2n} \in \R $, 
$ \varepsilon \in (0,\infty) $, 
$ t \in [0,T] $
it holds that 
\begin{equation}
\label{eq:lb.even}
\begin{split}
&
  \E\Big[\big\|
    X^{ 2n,\mathbf{u}^{\varepsilon,m,(\delta_1,\delta_2,\ldots,\delta_{2n})}_{2n,N} }_t
  \big\|^2_{H_{-r}}\Big]
\\&\geq
  \left[
  \frac{
      N^m
      \big(
      \prod\nolimits^{n-1}_{ i=0 }
      (1-2i)
      \big)
  }{
    c^r e^{ ct }
  }
  \right]^2
  \int^t_0
  \frac{
    e^{-8cn^3s}
    \| e^{2n^3sA} v^{0,-(\nicefrac{1}{4})-n\varepsilon+\sum^{2n}_{i=1}(\nicefrac{\delta_i}{2})}_{N^m,N} \|^4_H
  }{
    (n+n^2)^{8(|(\nicefrac{(\delta_{2n-1}+\delta_{2n})}{2})-(\nicefrac{1}{4})-\varepsilon|+\sum^{n-1}_{i=1}|(\nicefrac{(\delta_{2i-1}+\delta_{2i})}{2})-\varepsilon|)}
  }
  \,\diffns{s}
\\&\geq
  \left[
  \frac{
      N^m
      \big(
      \prod\nolimits^{n-1}_{ i=0 }
      (1-2i)
      \big)
  }{
    c^r e^{ c(1+4n^3)t }
    (2n^2)^{(1+4n\varepsilon+2\sum^{2n}_{i=1}|\delta_i|)}
  }
  \right]^2
  \int^t_0
  \| e^{2n^3sA} v^{0,-(\nicefrac{1}{4})-n\varepsilon+\sum^{2n}_{i=1}(\nicefrac{\delta_i}{2})}_{N^m,N} \|^4_H
  \,\diffns{s}
  .
\end{split}
\end{equation}
This and~\eqref{eq:lb.odd} yield that for all 
$ N, n \in \N $, 
$r\in[0,\infty)$, 
$ \delta_1,\delta_2,\ldots,\delta_n \in \R $, 
$ \varepsilon \in (0,\nicefrac{1}{4}) $, 
$m\in\N_0\cap[\frac{1}{4\varepsilon}-1,\infty)$,
$ t \in (0,T] $ 
it holds that 
\begin{equation}
\label{eq:pre.combined.lb}
\begin{split}
&
  \E\Big[\big\|
    X^{ n,\mathbf{u}^{\varepsilon,m,(\delta_1,\delta_2,\ldots,\delta_n)}_{n,N} }_t
  \big\|^2_{H_{-r}}\Big]
\\&\geq
  \left[
    \frac{
      N^m
      \big(
      \prod\nolimits^{ \lceil \nicefrac{n}{2} \rceil - 1 }_{ i=0 }
      (1-2i)
      \big)
    }{
  c^r e^{ c(1+4|\lceil n/2 \rceil|^3) t}
  (2|\lceil n/2 \rceil|^2)^{(1+4\lceil n/2 \rceil\varepsilon+2\sum^n_{i=1}|\delta_i|)} \,
  [ 1 + \sup_{s\in[0,T]} \| P e^{sA} v^{\lceil n/2 \rceil,-\varepsilon}_{\lceil n/2 \rceil N^m,N} \|^2_H ]^{
    ( \lceil \nicefrac{n}{2} \rceil - \nicefrac{1}{2} )
  }
    }
  \right]^2
\\&\quad\cdot
  \int^t_0
  \big\| e^{2|\lceil\nicefrac{n}{2}\rceil|^3sA} v^{0,-(\nicefrac{1}{4})-\lceil\nicefrac{n}{2}\rceil\varepsilon+\sum^n_{i=1}(\nicefrac{\delta_i}{2})}_{N^m,N} \big\|^4_H
  \,\diffns{s}
  .
\end{split}
\end{equation}
Moreover, note that for all 
$N,n\in\N$, 
$i\in\{1,2,\ldots,n\}$,
$\varepsilon\in(0,\frac{1}{4})$,
$m\in\N_0\cap[\frac{1}{4\varepsilon}-1,\infty)$,
$t\in[0,T]$ 
it holds that 
\begin{equation}
\begin{split}
&
  \|P e^{tA} v^{i,-\varepsilon}_{nN^m,N}\|^2_H
=
  \Big\|
  P e^{tA} (-A)^{-\varepsilon}
  \Big[
  {\smallsum^N_{j=1}e_{i+jnN^m}}
  \Big]
  \Big\|^2_H
\\&=
  \Big\|
  {\smallsum^N_{j=1}}
  \Big[
  e^{tA} (-A)^{-\varepsilon} e_{i+jnN^m}
  \Big]
  \Big\|^2_H
  =
  \sum^N_{j=1}
  \|
  e^{tA} (-A)^{-\varepsilon} e_{i+jnN^m}
  \|^2_H
\\&=
  \frac{1}{c^{2\varepsilon}}
  \left[
  \sum^N_{j=1}
  \frac{e^{-2tc(i+jnN^m)^2}}{
    (i+jnN^m)^{4\varepsilon}
  }
\right]
  \leq
  \frac{1}{c^{2\varepsilon}}
  \left[
  \sum^N_{j=1}
  \frac{1}{
    (jN^m)^{4\varepsilon}
  }
  \right]
  =
  \frac{1}{c^{2\varepsilon} N^{4m\varepsilon}}
  \left[
  1+
  \sum^N_{j=2}
  \frac{1}{
    j^{4\varepsilon}
  }
  \right]
\\&=
  \frac{1}{c^{2\varepsilon} N^{4m\varepsilon}}
  \bigg(
  1+
  \sum^N_{j=2}
  \int^j_{j-1}
  \frac{1}{
    j^{4\varepsilon}
  } \,
  dx
  \bigg)
  \leq
  \frac{1}{c^{2\varepsilon} N^{4m\varepsilon}}
  \bigg(
  1+
  \sum^N_{j=2}
  \int^j_{j-1}
  \frac{1}{
    x^{4\varepsilon}
  } \,
  dx
  \bigg)
\\&=
  \frac{1}{c^{2\varepsilon} N^{4m\varepsilon}}
  \bigg(
  1+
  \int^N_1
  \frac{1}{
    x^{4\varepsilon}
  } \,
  dx
  \bigg)
  =
  \frac{1}{c^{2\varepsilon} N^{4m\varepsilon}}
  \bigg(
  1+
  \frac{1}{(1-4\varepsilon)}
  \big(
  N^{(1-4\varepsilon)}-1
  \big)
  \bigg)
\\&=
  \frac{1}{c^{2\varepsilon}}
  \bigg(
  \frac{1}{N^{4m\varepsilon}}+
  \frac{1}{(1-4\varepsilon)}
  \bigg[
  N^{(1-4\varepsilon-4m\varepsilon)}-\frac{1}{N^{4m\varepsilon}}
  \bigg]
  \bigg)
  .
\end{split}
\end{equation}
This and the fact that
$
  \forall \,
  \varepsilon\in(0,\frac{1}{4}), \,
  m\in\N_0\cap[\frac{1}{4\varepsilon}-1,\infty)
  \colon
  1-4\varepsilon-4m\varepsilon\leq 0
$
ensure that for all 
$N,n\in\N$, 
$i\in\{1,2,\ldots,n\}$,
$\varepsilon\in(0,\frac{1}{4})$,
$m\in\N_0\cap[\frac{1}{4\varepsilon}-1,\infty)$
it holds that 
\begin{equation}
\label{eq:integrand.component.finite}
\begin{split}
  \sup_{t\in[0,T]}
  \|P e^{tA} v^{i,-\varepsilon}_{nN^m,N}\|^2_H
\leq
  \frac{1}{c^{2\varepsilon}}
  \bigg(
  \frac{1}{N^{4m\varepsilon}}+
  \frac{1}{(1-4\varepsilon)}-
  \frac{1}{(1-4\varepsilon)N^{4m\varepsilon}}
  \bigg)
  \leq
  \frac{1}{c^{2\varepsilon}(1-4\varepsilon)}
  .
\end{split}
\end{equation}
Plugging~\eqref{eq:integrand.component.finite} into the right hand side of~\eqref{eq:pre.combined.lb} yields that for all 
$ N, n \in \N $, 
$r\in[0,\infty)$, 
$ \delta_1,\delta_2,\ldots,\delta_n \in \R $, 
$ \varepsilon \in (0,\nicefrac{1}{4}) $, 
$m\in\N_0\cap[\frac{1}{4\varepsilon}-1,\infty)$,
$ t \in (0,T] $ 
it holds that 
\begin{equation}
\label{eq:combined.lb}
\begin{split}
&
  \E\Big[\big\|
    X^{ n,\mathbf{u}^{\varepsilon,m,(\delta_1,\delta_2,\ldots,\delta_n)}_{n,N} }_t
  \big\|^2_{H_{-r}}\Big]
\\&\geq
  \left[
    \frac{
      N^m
      \big(
      \prod\nolimits^{ \lceil \nicefrac{n}{2} \rceil - 1 }_{ i=0 }
      (1-2i)
      \big)
    }{
  c^r e^{ c(1+4|\lceil n/2 \rceil|^3) t}
  (2|\lceil n/2 \rceil|^2)^{(1+4\lceil n/2 \rceil\varepsilon+2\sum^n_{i=1}|\delta_i|)} \,
  [ 1 + \frac{1}{c^{2\varepsilon}(1-4\varepsilon)} ]^{
    ( \lceil \nicefrac{n}{2} \rceil - \nicefrac{1}{2} )
  }
    }
  \right]^2
\\&\quad\cdot
  \int^t_0
  \big\| e^{2|\lceil\nicefrac{n}{2}\rceil|^3sA} v^{0,-(\nicefrac{1}{4})-\lceil\nicefrac{n}{2}\rceil\varepsilon+\sum^n_{i=1}(\nicefrac{\delta_i}{2})}_{N^m,N} \big\|^4_H
  \,\diffns{s}
  .
\end{split}
\end{equation}
Next note that for all 
$ N,l \in \N $, 
$ m \in \N_0 $, 
$ \varepsilon \in (0,\infty) $, 
$ \delta \in [\nicefrac{1}{2}+2l\varepsilon,\infty) $, 
$ t \in (0,T] $
it holds that 
\begin{equation}
\begin{split}
&
  \int^t_0
  \big\| e^{2l^3sA} v^{0,-(\nicefrac{1}{4})-l\varepsilon+(\nicefrac{\delta}{2})}_{N^m,N} \big\|^4_H
  \,\diffns{s}
  =
  \int_0^t
  \left\|
  e^{ 2l^3sA }
  (-A)^{-(\nicefrac{1}{4})-l\varepsilon+(\nicefrac{\delta}{2})}
  \Big(
  {\smallsum_{ j = 1 }^N}
  e_{jN^m}
  \Big)
  \right\|^4_H
  ds
\\&=
  \int_0^t
  \left[
  \left\|
  {\smallsum_{ j = 1 }^N}
  \big[
  e^{ 2l^3sA }
  (-A)^{-(\nicefrac{1}{4})-l\varepsilon+(\nicefrac{\delta}{2})} e_{jN^m}
  \big]
  \right\|^2_H
  \right]^2
  ds
\\&=
  \int_0^t
  \left[
  \sum_{ j = 1 }^N
  \left\|
  e^{ 2l^3sA }
  (-A)^{-(\nicefrac{1}{4})-l\varepsilon+(\nicefrac{\delta}{2})} e_{jN^m}
  \right\|^2_H
  \right]^2
  ds
\\&=
  \int_0^t
  \left[
    {\sum_{ j = 1 }^N}
    \left(
      e^{ -2l^3cj^2N^{2m} s }
      \big\| (-A)^{-(\nicefrac{1}{4})-l\varepsilon+(\nicefrac{\delta}{2})} e_{jN^m} \big\|_H
    \right)^2
  \right]^2
  ds
\\ & =
  \sum_{ j, k = 1 }^N
      \| (-A)^{-(\nicefrac{1}{4})-l\varepsilon+(\nicefrac{\delta}{2})} e_{jN^m} \|^2_H \,
      \| (-A)^{-(\nicefrac{1}{4})-l\varepsilon+(\nicefrac{\delta}{2})} e_{kN^m} \|^2_H
  \left[
  \int_0^t
    e^{ -4l^3c ( j^2 + k^2 )N^{2m} s }
    \,
  ds
  \right]
\\ & =
  \sum_{ j, k = 1 }^{ N }
    \left(
    \left[
    \frac{
      (
        1 - e^{ -4l^3c (j^2+k^2) N^{2m} t }
      )
    }{
      4l^3c^{(2+4l\varepsilon-2\delta)} (j^2 + k^2) N^{2m}
    } 
    \right]
    (jN^m)^{ ( 2\delta -1 -4l\varepsilon ) }
    \,
    (kN^m)^{ ( 2\delta -1 -4l\varepsilon ) }
    \right)
  .
\end{split}
\end{equation}
The fact that 
$
  \forall \,
  l \in \N, \,
  \varepsilon \in (0,\infty), \,
  \delta \in [\nicefrac{1}{2}+2l\varepsilon,\infty)
  \colon
  2\delta-1-4l\varepsilon \geq 0
$
therefore assures that for all 
$ N,l \in \N $, 
$ m \in \N_0 $, 
$ \varepsilon \in (0,\infty) $, 
$ \delta \in [\nicefrac{1}{2}+2l\varepsilon,\infty) $, 
$ t \in (0,T] $
it holds that 
\begin{equation}
\label{eq:lb.integral}
\begin{split}
  \int^t_0
  \big\| e^{2l^3sA} v^{0,-(\nicefrac{1}{4})-l\varepsilon+(\nicefrac{\delta}{2})}_{N^m,N} \big\|^4_H
  \,\diffns{s}
&\geq
  \sum_{ j, k = 1 }^{ N }
    \frac{
      (
        1 - e^{ -4l^3c (j^2+k^2) N^{2m} t }
      )
    }{
      4l^3c^{(2+4l\varepsilon-2\delta)} (j^2 + k^2) N^{2m}
    } 
\\&\geq
  \frac{
      (
        1 - e^{ -ct }
      )
  }{ 4l^3c^{(2+4l\varepsilon-2\delta)} N^{2m}}
  \left[
  \sum_{ j, k = 1 }^N
    \frac{1}{
      (j^2+k^2)
    }
  \right]
  .
\end{split}
\end{equation}
This and~\eqref{eq:combined.lb} imply that for all 
$ n \in \N $, 
$r\in[0,\infty)$, 
$ \delta_1,\delta_2,\ldots,\delta_n \in \R $, 
$ \varepsilon \in (0,\nicefrac{1}{4}) $, 
$m\in\N_0\cap[\frac{1}{4\varepsilon}-1,\infty)$,
$ t \in (0,T] $ 
with 
$
  \sum^n_{i=1}
  \delta_i
  \geq
  \frac{1}{2}
  +
  2\lceil\nicefrac{n}{2}\rceil\varepsilon
$
it holds that 
\begin{equation}
\label{eq:numerator.infinite}
\begin{split}
&
  \sup_{N\in\N}
  \big(\E\big[\|
    X^{ n,\mathbf{u}^{\varepsilon,m,(\delta_1,\delta_2,\ldots,\delta_n)}_{n,N} }_t
  \|^2_{H_{-r}}\big]\big)^{\nicefrac{1}{2}}
\\&\geq
    \sup_{N\in\N}\bigg(
    \frac{
      N^m\prod\nolimits^{ \lceil \nicefrac{n}{2} \rceil - 1 }_{ i=0 }
      |1-2i|
    }{
  c^r e^{ c(1+4|\lceil n/2 \rceil|^3) t}
  (2|\lceil n/2 \rceil|^2)^{(1+4\lceil n/2 \rceil\varepsilon+2\sum^n_{i=1}|\delta_i|)}
  [ 1 + \frac{1}{c^{2\varepsilon}(1-4\varepsilon)} ]^{
    ( \lceil \nicefrac{n}{2} \rceil - \nicefrac{1}{2} )
  }
    }
\\&\quad\cdot
  \bigg[
  \frac{
      (
        1 - e^{ -ct }
      )
  }{ 4|\lceil\nicefrac{n}{2}\rceil|^3c^{(2+4\lceil\nicefrac{n}{2}\rceil\varepsilon-2\sum^n_{i=1}\delta_i)} N^{2m} } 
  \sum_{ j, k = 1 }^N
    \frac{1}{
      (j^2+k^2)
    }
  \bigg]^{\nicefrac{1}{2}}
  \bigg)
\\&=
    \frac{
      \prod\nolimits^{ \lceil \nicefrac{n}{2} \rceil - 1 }_{ i=0 }
      |1-2i|
    }{
  c^r e^{ c(1+4|\lceil n/2 \rceil|^3) t}
  (2|\lceil n/2 \rceil|^2)^{(1+4\lceil n/2 \rceil\varepsilon+2\sum^n_{i=1}|\delta_i|)}
  [ 1 + \frac{1}{c^{2\varepsilon}(1-4\varepsilon)} ]^{
  	( \lceil \nicefrac{n}{2} \rceil - \nicefrac{1}{2} )
  }
    }
\\&\quad\cdot
  \bigg[
  \frac{
      (
        1 - e^{ -ct }
      )
  }{ 4|\lceil\nicefrac{n}{2}\rceil|^3c^{(2+4\lceil\nicefrac{n}{2}\rceil\varepsilon-2\sum^n_{i=1}\delta_i)} } 
  \sum_{ j, k = 1 }^\infty
    \frac{1}{
      (j^2+k^2)
    }
  \bigg]^{\nicefrac{1}{2}}
  =
  \infty
  .
\end{split}
\end{equation}
Next note that~\eqref{eq:def.u.even} and~\eqref{eq:def.u.odd.combined} ensure that for all 
$ n \in \N $, 
$ \delta_1,\delta_2,\ldots,\delta_n \in \R $, 
$ \varepsilon \in (0,\nicefrac{1}{4}) $, 
$m\in\N_0\cap[\frac{1}{4\varepsilon}-1,\infty)$
it holds that 
\begin{equation}
\label{eq:project.multiplication}
\begin{split}
&
  \sup_{N\in\N}
  \left[
  \prod^n_{i=1}
  \| \theta^{n+1}_{i+1}(\mathbf{u}^{\varepsilon,m,(\delta_1,\delta_2,\ldots,\delta_n)}_{n,N}) \|_{H_{-\delta_i}}
  \right]
\\&=
  \sup_{N\in\N}
  \bigg(
  N^m \,
  \| v^{\lceil \nicefrac{n}{2} \rceil,\delta_n-(\nicefrac{1}{2})-\varepsilon}_{\lceil \nicefrac{n}{2} \rceil N^m,N} \|_{H_{-\delta_n}}
  \prod^{n-1}_{i=1}
  \| v^{\lceil \nicefrac{i}{2} \rceil,\delta_i-\varepsilon}_{\lceil \nicefrac{n}{2} \rceil N^m,N} \|_{H_{-\delta_i}}
  \bigg)
\\&\leq
  \bigg[
  \sup_{N\in\N}
  \Big(
  N^m \,
  \| v^{\lceil \nicefrac{n}{2} \rceil,\delta_n-(\nicefrac{1}{2})-\varepsilon}_{\lceil \nicefrac{n}{2} \rceil N^m,N} \|_{H_{-\delta_n}}
  \Big)\bigg] \,
  \prod^{n-1}_{i=1}
  \bigg[
  \sup_{N\in\N}
  \| v^{\lceil \nicefrac{i}{2} \rceil,\delta_i-\varepsilon}_{\lceil \nicefrac{n}{2} \rceil N^m,N} \|_{H_{-\delta_i}}
  \bigg]
\\&=
  \bigg[
  \sup_{N\in\N}
  \Big(
  N^{2m} \,
  \| v^{\lceil \nicefrac{n}{2} \rceil,\delta_n-(\nicefrac{1}{2})-\varepsilon}_{\lceil \nicefrac{n}{2} \rceil N^m,N} \|^2_{H_{-\delta_n}}
  \Big)
  \bigg]^{\nicefrac{1}{2}} \,
  \prod^{n-1}_{i=1}
  \bigg[
  \sup_{N\in\N}
  \| v^{\lceil \nicefrac{i}{2} \rceil,\delta_i-\varepsilon}_{\lceil \nicefrac{n}{2} \rceil N^m,N} \|^2_{H_{-\delta_i}}
  \bigg]^{\nicefrac{1}{2}}
  .
\end{split}
\end{equation}
Furthermore, observe that~\eqref{eq:def.v} shows that for all 
$n\in\N$, 
$m\in\N_0$, 
$\delta\in\R$, 
$\varepsilon\in(0,\infty)$
it holds that 
\begin{equation}
\label{eq:big.norm}
\begin{split}
&
  \sup_{N\in\N}
  \Big(
  N^{2m} \,
  \|v^{n,\delta-(\nicefrac{1}{2})-\varepsilon}_{nN^m,N}\|^2_{H_{-\delta}}
  \Big)
  =
  \sup_{N\in\N}
  \Big(
  N^{2m} \,
  \|v^{n,-(\nicefrac{1}{2})-\varepsilon}_{nN^m,N}\|^2_H
  \Big)
\\&=
  \sup_{N\in\N}
  \bigg(
  N^{2m} \,
  \bigg\|
  {\smallsum\limits^N_{j=1}}
  \big[
    (-A)^{-(\nicefrac{1}{2})-\varepsilon} e_{n+jnN^m}
  \big]
  \bigg\|^2_H
  \bigg)
\\&=
  \sup_{N\in\N}
  \left(
  \frac{N^{2m}}{
    c^{(1+2\varepsilon)}
  }
  \Bigg[
  \sum^N_{j=1}
  \frac{1}{
    (n+jnN^m)^{(2+4\varepsilon)}
  }
  \Bigg]
  \right)
\leq
  \sup_{N\in\N}
  \Bigg(
  \frac{N^{2m}}{
    c^{(1+2\varepsilon)}
  }
  \Bigg[
  \sum^\infty_{j=1}
  \frac{1}{
    (jN^m)^{(2+4\varepsilon)}
  }
  \Bigg]
  \Bigg)
\\&=
  \sup_{N\in\N}
  \Bigg(
  \frac{1}{
    N^{4m\varepsilon} \,
    c^{(1+2\varepsilon)}
  }
  \Bigg[
  \sum^\infty_{j=1}
  \frac{1}{
    j^{(2+4\varepsilon)}
  }
  \Bigg]
  \Bigg)
=
  \frac{1}{
    c^{(1+2\varepsilon)}
  }
  \Bigg[
  \sum^\infty_{j=1}
  \frac{1}{
    j^{(2+4\varepsilon)}
  }
  \Bigg]
  < \infty.
\end{split}
\end{equation}
In addition, note that~\eqref{eq:def.v} and~\eqref{eq:integrand.component.finite} imply that for all 
$n\in\N$,
$i\in\{1,2,\ldots,n\}$, 
$\delta\in\R$, 
$\varepsilon\in(0,\frac{1}{4})$, 
$m\in\N_0\cap[\frac{1}{4\varepsilon}-1,\infty)$
it holds that 
\begin{equation}
\label{eq:small.norm}
\begin{split}
&
  \sup_{N\in\N}
  \|v^{i,\delta-\varepsilon}_{nN^m,N}\|^2_{H_{-\delta}}
  =
  \sup_{N\in\N}
  \|v^{i,-\varepsilon}_{nN^m,N}\|^2_H
=
  \sup_{N\in\N}
  \|Pv^{i,-\varepsilon}_{nN^m,N}\|^2_H
  \leq
  \frac{1}{c^{2\varepsilon}(1-4\varepsilon)}
  < \infty
  .
\end{split}
\end{equation}
Combining~\eqref{eq:project.multiplication} with~\eqref{eq:big.norm} and~\eqref{eq:small.norm} yields that for all 
$ n \in \N $, 
$ \delta_1,\delta_2,\ldots,\delta_n \in \R $, 
$ \varepsilon \in (0,\nicefrac{1}{4}) $, 
$m\in\N_0\cap[\frac{1}{4\varepsilon}-1,\infty)$
it holds that 
\begin{equation}
\label{eq:denom.finite}
\begin{split}
&
  \sup_{N\in\N}
  \left[
  \prod^n_{i=1}
  \| \theta^{n+1}_{i+1}(\mathbf{u}^{\varepsilon,m,(\delta_1,\delta_2,\ldots,\delta_n)}_{n,N}) \|_{H_{-\delta_i}}
  \right]
  < \infty.
\end{split}
\end{equation}
Moreover, note that for all 
$ N,n \in \N $, 
$ k \in \N_0 $, 
$ r \in \R $ 
it holds that 
$
v^{k,r}_{n,N}
\in \operatorname{span}(\{e_m\colon m\in\N\}) \setminus \{0\}
$. 
This and the fact that 
$
\operatorname{span}(\{e_n\colon n\in\N\}) 
\subseteq
\cap_{r\in\R} H_r
$
ensure that for all 
$ N,n\in\N $, 
$ m \in \N_0 $, 
$ \varepsilon \in \R $, 
$ \boldsymbol{\delta} \in \R^n $ 
it holds that 
\begin{equation}
\label{eq:nonzero.vector}
\mathbf{u}^{\varepsilon,m,\boldsymbol{\delta}}_{n,N}
\in \big( (\cap_{r\in\R} H_r) \setminus \{0\} \big)^{n+1}.
\end{equation}
Combining this with~\eqref{eq:denom.finite} assures that for all 
$ n \in \N $, 
$ \boldsymbol{\delta} \in \R^n $, 
$ \varepsilon \in (0,\frac{1}{4}) $, 
$m\in\N_0\cap[\frac{1}{4\varepsilon}-1,\infty)$
it holds that 
\begin{equation}
\label{eq:denom.convergence}
  \inf_{N\in\N}
  \left[
  \frac{1}{
  \prod^n_{i=1}
  \| \theta^{n+1}_{i+1}(\mathbf{u}^{\varepsilon,m,\boldsymbol{\delta}}_{n,N}) \|_{H_{-\delta_i}}
  }
  \right]
  =
  \frac{1}{
    \big[
    \sup_{N\in\N}
  \big(
  \prod^n_{i=1}
  \| \theta^{n+1}_{i+1}(\mathbf{u}^{\varepsilon,m,\boldsymbol{\delta}}_{n,N}) \|_{H_{-\delta_i}}
  \big)
  \big]
  }
  \in(0,\infty).
\end{equation}
This, \eqref{eq:numerator.infinite}, and~\eqref{eq:nonzero.vector} show that for all 
$ n \in \N $, 
$ q \in [0,\infty) $, 
$ \boldsymbol{\delta}=(\delta_1, \delta_2, \ldots, \delta_n) \in \R^n $,
$ \varepsilon \in (0,\frac{\lceil\nicefrac{n}{2}\rceil}{2}) $, 
$m\in\N_0\cap[\frac{\lceil\nicefrac{n}{2}\rceil}{2\varepsilon}-1,\infty)$, 
$ t \in (0,T] $ 
with 
$
  \sum^n_{i=1} \delta_i
  \geq
  \frac{1}{2} + \varepsilon
$
it holds that 
\begin{equation}
\label{eq:infinite}
\begin{split}
&
  \sup_{ \mathbf{u}=(u_0,u_1,\ldots,u_n) \in (\nzspace{(\cap_{r\in\R}H_r)})^{n+1} }
  \left[
  \frac{
    \big(\E\big[\| X^{ n,\mathbf{u} }_t \|^2_{H_{-q}}\big]\big)^{\nicefrac{1}{2}}
  }{
    \prod^n_{ i=1 }
    \|u_i\|_{H_{-\delta_i}}
  }
  \right]
\\&=
  \sup_{N\in\N}
  \sup_{ \mathbf{u}=(u_0,u_1,\ldots,u_n) \in (\nzspace{(\cap_{r\in\R}H_r)})^{n+1} }
  \left[
  \frac{
    \big(\E\big[\| X^{ n,\mathbf{u} }_t \|^2_{H_{-q}}\big]\big)^{\nicefrac{1}{2}}
  }{
    \prod^n_{ i=1 }
    \|\theta^{n+1}_{i+1}(\mathbf{u})\|_{H_{-\delta_i}}
  }
  \right]
\\&\geq
  \sup_{ N \in \N }
  \left[
  \frac{
    \big(\E\big[\| X^{ n,\mathbf{u}^{\nicefrac{\varepsilon}{(2\lceil\nicefrac{n}{2}\rceil)},m,\boldsymbol{\delta}}_{n,N} }_t \|^2_{H_{-q}}\big]\big)^{\nicefrac{1}{2}}
  }{
    \prod^n_{i=1}
    \| \theta^{n+1}_{i+1}(\mathbf{u}^{\nicefrac{\varepsilon}{(2\lceil\nicefrac{n}{2}\rceil)},m,\boldsymbol{\delta}}_{n,N}) \|_{H_{-\delta_i}}
  }
  \right]
\\&\geq
  \left[
  \inf_{N\in\N}
  \frac{1}{
    \prod^n_{i=1}
    \| \theta^{n+1}_{i+1}(\mathbf{u}^{\nicefrac{\varepsilon}{(2\lceil\nicefrac{n}{2}\rceil)},m,\boldsymbol{\delta}}_{n,N}) \|_{H_{-\delta_i}}
  }
  \right]
  \left[
  \sup_{ N \in \N }
    \big(\E\big[\| X^{ n,\mathbf{u}^{\nicefrac{\varepsilon}{(2\lceil\nicefrac{n}{2}\rceil)},m,\boldsymbol{\delta}}_{n,N} }_t \|^2_{H_{-q}}\big]\big)^{\nicefrac{1}{2}}
  \right]
  =\infty
  .
\end{split}
\end{equation}
This and H\"{o}lder's inequality establish item~\eqref{item:blowup}. 
The proof of Theorem~\ref{lem:blowup} is thus completed.
\end{proof}

\section*{Acknowledgements}

We gratefully acknowledge Adam Andersson for a number of useful comments.
This project has been supported through the SNSF-Research project 200021\_156603 
"Numerical approximations of nonlinear stochastic ordinary and partial differential equations".

\bibliographystyle{acm}
\bibliography{Bib/bibfile}
\end{document}